\documentclass[final,leqno]{siamltex}
\usepackage{amsmath}
\usepackage{amssymb}
\usepackage{subfig}
\usepackage{multirow}
\usepackage{slashbox}
\usepackage{graphicx}
\newcommand{\bR}{\mathbf{R}}
\newcommand{\bx}{\mathbf{x}}
\newcommand{\bb}{\mathbf{b}}
\newcommand{\bu}{\mathbf{u}}

\newcommand{\bv}{\mathbf{v}}

\newcommand{\bw}{\mathbf{w}}

\newcommand{\bbf}{\mathbf{f}}
\newcommand{\bg}{\mathbf{g}}
\newcommand{\p}{\partial}

\newcommand{\CV}{C_{\mbox{\tiny V}}}
\newcommand{\CW}{C_{\mbox{\tiny W}}}
\newcommand{\tCV}{\widetilde{C}_{\mbox{\tiny V}}}
\newcommand{\tCW}{\widetilde{C}_{\mbox{\tiny W}}}

\newcommand{\ome}{\omega}
\newcommand{\deltaV}{\delta_{\mbox{\tiny V}}}
\newcommand{\deltaW}{\delta_{\mbox{\tiny W}}}
\newcommand{\omeV}{\ome_{\mbox{\tiny V}}}
\newcommand{\omeW}{\ome_{\mbox{\tiny W}}}
\newcommand{\Ome}{\Omega}
\newcommand{\nab}{\nabla}

\newcommand{\cE}{\mathcal{E}}
\newcommand{\cG}{\mathcal{G}}
\newcommand{\Emu}{\cE_{\mbox{\tiny mu}}}

\newcommand{\Esy}{\cE_{\mbox{\tiny sy}}}
\newcommand{\cI}{\mathcal{I}}
\newcommand{\cA}{\mathcal{A}}

\newcommand{\cF}{\mathcal{F}}
\newcommand{\cL}{\mathcal{L}}

\newcommand{\cR}{\mathcal{R}}
\newcommand{\cS}{\mathcal{S}}

\newcommand{\cP}{\mathcal{P}}
\newcommand{\Pad}{\cP_{\mbox{\tiny ad}}}
\newcommand{\Pmu}{\cP_{\mbox{\tiny mu}}}
\newcommand{\Psy}{\cP_{\mbox{\tiny sy}}}
\newcommand{\cQ}{\mathcal{Q}}
\newcommand{\Qad}{\cQ_{\mbox{\tiny ad}}}

\newcommand{\tP}{\widetilde{P}}
\newcommand{\ctP}{\widetilde{\cP}}
\newcommand{\tQ}{\widetilde{Q}}
\newcommand{\ctQ}{\widetilde{\cQ}}

\newcommand{\Div}{{\rm div\ }}

\newcommand{\normV}[1]{\Vert#1\Vert_{\mbox{\tiny V}}}
\newcommand{\normW}[1]{\Vert#1\Vert_{\mbox{\tiny W}}}
\newcommand{\normVj}[1]{\Vert#1\Vert_{\mbox{\tiny V$_j$}}}
\newcommand{\normWj}[1]{\Vert#1\Vert_{\mbox{\tiny W$_j$}}}

\newcommand{\normWnot}[1]{\Vert#1\Vert_{\mbox{\tiny W$_0$}}}

\newtheorem{remark}{Remark}[section]
\begin{document}
\title{On Schwarz Methods for Nonsymmetric and Indefinite Problems} 
\markboth{X. FENG and C. LORTON}{SCHWARZ METHODS FOR NONSYMMETRIC AND
INDEFINITE PROBLEMS}

\author{
Xiaobing Feng\thanks{Department of Mathematics, The University of
Tennessee, Knoxville, TN 37996, U.S.A.  ({\tt xfeng@math.utk.edu}).
The work of this author was partially supported by the NSF grants DMS-0710831
and DMS-1016173.}
\and
Cody Lorton\thanks{Department of Mathematics, The University of
Tennessee, Knoxville, TN 37996, U.S.A.  ({\tt lorton@math.utk.edu}).
The work of this author was partially supported by the NSF grants DMS-0710831
and DMS-1016173.}
}

\maketitle

\begin{abstract}
In this paper we introduce a new Schwarz framework and theory, 
based on the well-known idea of space decomposition, for nonsymmetric 
and indefinite linear systems arising from continuous and discontinuous
Galerkin approximations of general nonsymmetric and indefinite elliptic
partial differential equations. The proposed Schwarz framework and theory 
are presented in a variational setting in Banach spaces 
instead of Hilbert spaces which is the case for the well-known symmetric
and positive definite (SPD) Schwarz framework and theory. 
Condition number estimates for the additive and hybrid Schwarz preconditioners 
are established. The main idea of our nonsymmetric and indefinite Schwarz 
framework and theory is to use weak coercivity (satisfied by the nonsymmetric 
and indefinite bilinear form) induced norms to replace the standard 
bilinear form induced norm in the SPD Schwarz framework and theory. 
Applications of the proposed nonsymmetric and indefinite Schwarz framework  
to solutions of discontinuous Galerkin approximations of 
convection-diffusion problems are also discussed.  Extensive 1-D numerical 
experiments are also provided to gauge the performance of the proposed 
Schwarz methods.
\end{abstract}

\begin{keywords}
Schwarz methods and preconditioners, domain decomposition, space 
decomposition, inf-sup condition, strong and weak coercivity, 
condition number estimates.
\end{keywords}

\begin{AMS}
65N55, 65F10
\end{AMS}


\section{Introduction} \label{sec-1}
The original Schwarz method, proposed and analyzed by Hermann 
Schwarz in 1870 \cite{Schwarz1870}, is an iterative method to find the solution 
of a partial differential equation (PDE) on a complicated domain which 
is the union of two overlapping simpler subdomains. The method solves
the equation on each of the two subdomains by using 
the latest values of the approximate solution as the boundary conditions
on the parts of the subdomain boundaries which are inside of the given
domain. The idea of splitting a given problem posed on a large (and possibly 
complicated) domain into several subproblems posed on smaller subdomains
and then solving the subdomain problems either sequentially or in parallel 
is a very appealing idea. Such a  ``divide-and-conquer" idea is
at the heart of every domain decomposition or Schwarz method.

It is well-known that \cite{TW05} the domain decomposition strategy 
can be introduced at the following three different levels: the continuous
level for PDE analysis as proposed and analyzed by Hermann Schwarz in 1870, 
the discretization level for constructing 
(hybrid and composite) discretization methods, and the algebraic level
for solving algebraic systems arising from the numerical approximations 
of PDE problems. These three levels are often interconnected, and each
of them has its own merit to be studied.  Most of the recent
efforts and attentions have been focused on the algebraic level. 
The field of domain decomposition methods has blossomed and 
undergone intensive and phenomenal development during the last 
thirty years (cf. \cite{SBG96,QV99,TW05} and the references therein). 
The phenomenal development   
has largely been driven by the ever increasing demands for fast solvers
for solving important and complicated scientific, engineering and 
industrial application problems which are often governed mathematically
by a PDE or a system of PDEs. It has also been infused and 
facilitated by the rapid advances in computer hardware and 
the emergence of parallel computing technologies. 

At the algebraic level, domain decomposition methods
or Schwarz methods have been well developed and studied 
for various numerical approximations (discretizations)
of many types of PDE problems including finite element 
methods (cf. \cite{DW90,Xu92}), mixed finite element methods and 
spectral methods (cf. \cite{TW05}), and discontinuous Galerkin 
methods (cf. \cite{FK01,LT03,FK05,AA07}). A general abstract framework, backed by an elegant convergence theory, was well established 
many years ago for symmetric and positive definite (SPD) PDE problems
and their numerical approximations (cf. \cite{DW90,Xu92,SBG96,QV99,TW05,XZ02}
and the references therein).  

Despite the tremendous advances in domain decomposition (Schwarz) 
methods over the past thirty years, the current framework and 
convergence theory are mainly confined to SPD problems 
in Hilbert spaces. Because the framework and especially the 
convergence theory indispensably rely on the SPD properties of the underlying 
problem and the Hilbert space structures, they do not apply to 
genuinely nonsymmetric and/or indefinite problems. As a result, the SPD 
framework and theory leave many important and interesting problems uncovered
as pointed out in \cite[page 311]{TW05}.

This paper attempts to address this important issue 
in Schwarz methods. The goal of this paper is 
to introduce a new Schwarz framework and theory, based on the well-known 
idea of space decomposition as in the SPD case, for nonsymmetric
and indefinite linear systems arising from continuous and discontinuous
Galerkin approximations of general nonsymmetric and indefinite elliptic
partial differential equations under some ``minimum" 
structure assumptions. Unlike the SPD framework and theory,
our new framework and theory are presented in a 
variational setting in Banach spaces instead of Hilbert spaces. 
Such a general framework allows broader applications of Schwarz methods.
Both additive Schwarz and multiplicative as weel as hybrid Schwarz methods
are developed.  A comprehensive convergence  theory is provided 
which includes condition number estimates for the
additive Schwarz preconditioners and hybrid Schwarz preconditioners.
The main idea of our nonsymmetric and indefinite Schwarz framework and
theory is to use weak coercivity (satisfied by the nonsymmetric and
indefinite bilinear form) induced norms to replace the standard
bilinear form induced norm in the SPD Schwarz framework and theory
(see Sections \ref{sec-2}--\ref{sec-4} for a detailed exposition). 
As expected, working with such weak coercivity induced norms and 
nonsymmetric and indefinite bilinear forms is quite delicate.  It requires
new and different technical tools in order to establish our convergence theory.

The remainder of this paper is organized in the following way.  
In Section \ref{sec-2}, we introduce notation, the functional 
setting, and the variational problems which we aim to solve. 
Section 2 also contains some further discussions on the main idea of 
the paper.  Section \ref{sec-3} is devoted to establishing an abstract 
additive Schwarz, multiplicative Schwarz, and hybrid Schwarz
framework for general nonsymmetric and indefinite algebraic problems
in a variational setting in general Banach spaces.
In Section \ref{sec-4}, we present an abstract convergence theory for
the additive and hybrid Schwarz methods proposed in Section \ref{sec-3}. 
In Section \ref{sec-5}, we present some applications of the proposed
nonsymmetric and indefinite Schwarz framework 
to discontinuous Galerkin approximations of convection-diffusion 
(in particular, convection-dominated) problems.  We also provide 
extensive 1-D numerical experiments to gauge the performance of the 
proposed nonsymmetric and indefinite Schwarz methods.

\section{Functional setting and statement of problems} \label{sec-2}

\subsection{Variational problem}\label{sec-2.1}
Let $X$ be a real Hilbert space with the inner product $(\cdot,\cdot)_X$ 
and the induced norm $\|\cdot\|_X$. Let $V, W\subset X$ be two 
reflexive Banach spaces endowed with the norms $\|\cdot\|_V$ and 
$\|\cdot\|_W$ respectively. Let $\cA(\cdot,\cdot)$ be a real bilinear form 
defined on the product space $V\times W$ and $\cF$ be a real linear 
functional defined on $W$. We consider the following variational 
problem: Find $u\in V$ such that 
\begin{align}\label{eq2.1}
\cA(u,w)=\cF(w) \qquad\forall w\in W.
\end{align}
The well-posedness of the above variational problem has been extensively studied.
%
%
One of such results is summarized in the following theorem:

\begin{theorem}{\rm (cf. \cite{BA72})} \label{babuska}
Suppose that $\cF$ is a bounded linear functional on $W$. 
Assume that $\cA(\cdot,\cdot)$ is continuous and {\em weakly coercive}
in the sense that there exist constants $C_\cA,\gamma_\cA >0$ such that
\begin{alignat}{2}\label{eq2.6}
|\cA(v,w)| &\leq C_\cA  \|v\|_V \|w\|_W &&\qquad\forall v\in V,\, w\in W,\\
\sup_{w\in W} \frac{\cA(v,w)}{\|w\|_W}
&\geq \gamma_\cA \|v\|_V &&\qquad\forall v\in V, \label{eq2.7}\\
\sup_{v\in V} \cA(v,w) &>0  &&\qquad\forall\,0\neq w\in W. \label{eq2.8} 
\end{alignat}
Then problem \eqref{eq2.1} has a unique solution $u\in V$. Moreover,
\begin{align}\label{eq2.5}
\|u\|_V\leq \frac{\|\cF\|}{ \gamma_\cA}.
\end{align}

\end{theorem}

\begin{remark}\label{rem2.1}
%
(a) Theorem \ref{babuska} is called Lax-Milgram-Babu\v{s}ka 
theorem in the literature (cf. \cite{Rosca89}). It was first introduced to 
the finite element context in \cite{Babuska71} 
(also see \cite{BA72}). An earlier version of the theorem can also be 
found in \cite{Nirenberg55}.

(b) As pointed out in \cite[page 117]{BA72}, condition \eqref{eq2.8}
can be replaced by the following more restrictive condition: There exists 
a constant $\beta_\cA>0$ such that
\begin{align}\label{eq2.8a}
\sup_{v\in V} \frac{\cA(v,w)}{\|v\|_V}
\geq \beta_\cA \|w\|_W \qquad\forall w\in W. 
\end{align}
The above condition can be viewed as a {\em weak coercivity} condition 
for the adjoint bilinear form $\cA^*(\cdot,\cdot)$ of $\cA(\cdot,\cdot)$. 

(c) Weak coercivity condition \eqref{eq2.7} is often called the {\em inf-sup} 
or {\em Babu\v{s}ka--Brezzi} condition in the finite element
literature \cite{BS08,Ciarlet78} for a different reason. It appears 
and plays a vital role for saddle point problems 
and their (mixed) finite element approximations (cf. \cite{Brezzi74,BF91}).  

(d) Theorem \ref{babuska} is certainly valid when $V=W$. 
Since condition \eqref{eq2.7} is weaker than the strong coercivity,
then Theorem \ref{babuska} is a stronger result than the classical
Lax-Milgram Theorem for the case $V=W$. Indeed, for most convection-dominated 
convection-diffusion problems, $V=W$. However, there are situations where
condition \eqref{eq2.7} holds but strong coercivity fails. 

(e) There are also situations where one prefers to use different norms for 
the trial space $V$ and the test space $W$ even if $V=W$. Theorem \ref{babuska} 
also provides a convenient framework to handle such a situation.
\end{remark}

\subsection{Discrete problem}\label{sec-2.2}
As problem \eqref{eq2.1} is posed on infinite dimensional spaces $V$ 
and $W$, to solve it numerically, one must approximate $V$ and $W$ by
some finite dimensional spaces $V_n, W_n\subset X$.  Here $n=\mbox{dim}(V_n) 
=\mbox{dim}(W_n)$ is a positive integer which denotes the dimension
of $V_n$ and $W_n$. If one of (or both) $V_n$ and $W_n$ 
is not a subspace of its corresponding infinite dimensional space, then
one also needs to provide an approximate bilinear form $a(\cdot,\cdot)$ 
for $\cA(\cdot,\cdot)$ so that $a(\cdot,\cdot)$ is well defined 
on $V_n\times W_n$. In addition, if $W_n$ is not a subspace of $W$ 
one also needs to provide an approximate linear functional $f$ for $\cF$ 
so that $f$ is well defined on $W_n$. 

Once $V_n, W_n, a$ and $f$ are constructed, the Galerkin method 
for problem \eqref{eq2.1} is defined as seeking $u_n\in V_n$ such that
\begin{align}\label{eq2.10}
a(u_n,w_n)=f(w_n) \qquad \forall w_n\in W_n.
\end{align}

Pick a basis $\{\phi^{(j)}\}_{j=1}^n$ of $V$ and a basis 
$\{\psi^{(j)}\}_{j=1}^n$ of $W$.  It is trivial to check that the discrete
variational problem \eqref{eq2.10} can be rewritten as the following linear 
system problem:   
\begin{align}\label{eq2.11}
A\bu=\bbf,
\end{align}
where $\bu=[u^{(j)}]_{j=1}^n$ is the coefficient vector of the 
representation of $u_n$ in terms of the basis $\{\phi^{(j)}\}_{j=1}^n$ and
\begin{alignat}{2}\label{eq2.11a}
A&=\bigl[a_{ij}\bigr]_{i,j=1}^n,  &\qquad a_{ij} &=a(\phi^{(j)},\psi^{(i)}),\\
\bbf&=\bigl[f^{(i)}\bigr]_{i=1}^n, &\qquad f^{(i)} &= f(\psi^{(i)}).
\end{alignat}

The properties of matrix $A$ (called a stiffness matrix)
are obviously determined by the properties of the discrete bilinear
form $a(\cdot,\cdot)$ and the approximate spaces $V_n$ and $W_n$. 
When $V_n=W_n$ it is well known that \cite{GVL96} $A$ is symmetric 
if and only if $a(\cdot,\cdot)$ is symmetric and $A$ is positive 
definite provided that $a(\cdot,\cdot)$ is strongly coercive 
on $V_n\times V_n$. 
In general, $A$ is just an $n\times n$ nonsymmetric real matrix if 
$a(\cdot,\cdot)$ is not symmetric.  $A$ also can be indefinite (i.e., 
$A$ has at least one negative and one positive eigenvalue) if $a(\cdot,\cdot)$ 
fails to be coercive.  

As \eqref{eq2.11} is a square linear system, by a well-known 
algebraic fact we know that \eqref{eq2.11} has a unique solution $\bu$ 
provided that the stiffness matrix $A$ is nonsingular. This nonsingular
condition on $A$ becomes necessary if one wants \eqref{eq2.11} 
to be uniquely solvable for {\em arbitrary} vector $\bbf$.  For most 
application problems (such as boundary value problems for elliptic
PDEs), one needs to consider various choices of the ``load"
functional $\cF$, so the vector $\bbf$ is practically 
``arbitrary" in \eqref{eq2.11}. Hence, besides some deeper mathematical 
and algorithmic considerations, asking for the stiffness matrix $A$ to 
be {\em nonsingular} is a ``minimum" requirement for the discretization 
method \eqref{eq2.10} to be practically useful. 

Sufficient conditions on the discrete bilinear form $a(\cdot,\cdot)$ and 
the approximate spaces $V_n$ and $W_n$ which infer the unique solvability
of the linear system \eqref{eq2.11} have been well 
studied and understood in the past thirty years.  In particular, 
for the SPD type (algebraic) problems arising from various 
discretizations of boundary value problems for elliptic 
PDEs \cite{Babuska71,BA72,Ciarlet78,BS08,BA72,BM97,Riviere08}.  In the 
following we shall quote some of these well-known results in a theorem 
which is a counterpart of Theorem \ref{babuska}. 


\begin{theorem}{\rm (cf. \cite{Babuska71,BA72})} \label{discrete_babuska}
Suppose that $f$ is a bounded linear functional on $W_n$. 
Assume that $a(\cdot,\cdot)$ is continuous and 
weakly coercive in the sense that there exist constants 
$C_a,\gamma_a,\beta_a >0$ such that 
\begin{alignat}{2}\label{eq2.16}
|a(v,w)| &\leq C_a  \|v\|_{V_n} \|w\|_{W_n} &&\qquad\forall v\in V_n,
\, w\in W_n,\\
\sup_{w\in W_n} \frac{a(v,w)}{\|w\|_{W_n}}
&\geq \gamma_a \|v\|_{V_n} &&\qquad\forall v\in V_n, \label{eq2.17}\\
\sup_{v\in V_n} \frac{a(v,w)}{\|v\|_{V_n}}
&\geq \beta_a \|w\|_{W_n} &&\qquad\forall w\in W_n. \label{eq2.18}
\end{alignat}
Then problem \eqref{eq2.10} has a unique solution $u_n\in V_n$. Moreover,
\begin{align}\label{eq2.15}
\|u_n\|_{V_n}\leq \frac{M_f}{\gamma_a}
\end{align}
where $M_f$ is a positive constant.

\end{theorem}

A few remarks are in order about the above well-posedness theorem.

\begin{remark}\label{rem2.2}
%
(a) Condition \eqref{eq2.18} is equivalent to requiring that the adjoint
$a^*(\cdot,\cdot)$ of $a(\cdot,\cdot)$ is weakly coercive.

(b) Conditions \eqref{eq2.16}--\eqref{eq2.18} are analogies of 
their continuous counterparts \eqref{eq2.6}--\eqref{eq2.8}. 
Discrete weak coercivity condition \eqref{eq2.17} is often called the 
{\em inf-sup} or {\em Babu\v{s}ka--Brezzi} condition in the finite element
literature \cite{BS08,Ciarlet78} for a different reason. It is the 
most important one in a set of sufficient conditions for a mixed finite 
element to be stable (cf. \cite{Brezzi74,BF91}).

(c) A numerical method which fulfills conditions \eqref{eq2.16}--\eqref{eq2.18}
is guaranteed to be uniquely solvable and stable. Hence, these conditions can 
be used as a test stone to determine whether a numerical method is a 
``good" method. For this reason, we shall call the numerical 
method \eqref{eq2.10} an {\em inf-sup preserving method}
or a {\em weak coercivity preserving method} if it satisfies
\eqref{eq2.16}--\eqref{eq2.18}.

(d) Theorem \ref{discrete_babuska}
focuses on the unique solvability and the stability of the 
numerical method \eqref{eq2.10} not on the accuracy of the
method. We like to note that method \eqref{eq2.10} indeed is
an accurate numerical method provided that approximate spaces $V_n$ 
and $W_n$ are accurate approximations of $V$ and $W$ (cf. \cite{BA72}).
\end{remark}

\subsection{Main objective}
As we briefly explained above, approximating the variational 
problem \eqref{eq2.1} by a Galerkin method certainly
results in solving the linear system \eqref{eq2.11}. 
It is well known that the common dimension $n$ of the approximation 
spaces $V_n$ and $W_n$ has to be sufficiently large in order 
for the Galerkin method to be accurate.  As a result,
the size of the linear system (i.e., the size of the matrix $A$) is
expected to be very large in applications. 
Moreover, if \eqref{eq2.1} is a variational formulation 
of some elliptic boundary value problem, then the stiffness
matrix $A$ is certainly ill-conditioned in the sense that
the condition number $\kappa(A):=\|A\| \|A^{-1}\|$ is very large. 
Here $\|A\|$ denotes a matrix norm of $A$. For example, 
in the case of second and fourth order elliptic boundary value problems,
$\kappa(A)=O(n^{\frac{2}{d}})$ and $\kappa(A)=O(n^{\frac{4}{d}})$, respectively, 
where d is the spatial dimension of the domain. 
(cf. \cite{BS08,TW05}). Consequently, it is not efficient to solve 
linear system \eqref{eq2.11} directly using classical 
iterative methods even if they converge. Furthermore, unlike 
in the SPD case, classical iterative methods often do not converge 
for general nonsymmetric and indefinite linear system \eqref{eq2.11}
(cf. \cite{GVL96,TW05}).

As a first step toward developing better iterative solvers for 
nonsymmetric and indefinite linear system \eqref{eq2.11}, it is natural 
to design a ``good" preconditioner (i.e., an $n\times n$ real matrix $B$) 
such that $BA$ is well-conditioned (i.e., $\kappa(BA)$ is 
relatively small, say, significantly smaller than $\kappa(A)$).  Then one 
can try classical iterative methods.  In particular, the Generalized Minimal 
Residual (GMRES) method can be used on the preconditioned system
\begin{equation} \label{eq2.20}
BA \bx=B\bb.
\end{equation}
One can also develop some new (and hopefully better) iterative methods 
if classical iterative methods still do not work as well 
on \eqref{eq2.20} as one had hoped.

As was already mentioned in Section \ref{sec-1}, the focus of
this paper is exactly what is described above. Our goal is to 
develop a new Schwarz framework and theory, based on the well-known idea of 
space decomposition, for solving nonsymmetric and indefinite 
linear system \eqref{eq2.11} which arises from the Galerkin 
method \eqref{eq2.10} as an approximation of the variational 
problem \eqref{eq2.1}.  As expected, our nonsymmetric and indefinite
Schwarz framework and theory are natural extensions of the well-known
SPD Schwarz framework and theory which were nicely described 
in \cite{DW90,Xu92,SBG96,QV99,TW05}.

\section{An abstract Schwarz framework for nonsymmetric 
and indefinite problems} \label{sec-3}

For the sake of notational brevity, throughout the remainder of this paper we shall 
suppress the sub-index $n$ in the discrete spaces $V_n$ and $W_n$
and in discrete functions $u_n$, $v_n$ and $w_n$.  In other words, $V$ and 
$W$ are used to denote $V_n$ and $W_n$, and $u$, $v$ and $w$ are used
to denote $u_n$, $v_n$ and $w_n$.
In addition, we shall make an effort below
to use the same or similar terminologies, as well as space and norm notation 
as those in \cite{TW05} for the symmetric and positive definite (SPD)
Schwarz framework and theory. We shall also make
comments about notation and terminologies which have no
SPD counterparts and try to make links between the well
known SPD Schwarz framework and theory and our nonsymmetric
and indefinite Schwarz framework and theory.

To motivate, we recall that in the SPD Schwarz framework and theory
\cite{DW90,Xu92,SBG96,QV99,TW05}, since $V=W$ and the discrete bilinear 
form $a(\cdot,\cdot)$ is symmetric and strongly coercive, 
$\sqrt{a(v,v)}$ defines a convenient norm (which is also 
equivalent to the $\|\cdot\|_{V}$-norm) on the space $V$ (as well as 
on its subspaces). This bilinear form induced norm  
plays a vital role in the SPD Schwarz framework and theory. 

Unfortunately, without the symmetry and strong coercivity assumptions
on $a(\cdot,\cdot)$, $\sqrt{a(v,v)}$ is not a norm anymore when $V=W$.  
It is not even well defined if $V\neq W$! To overcome this 
difficulty, the existing nonsymmetric and indefinite Schwarz framework
and theory (cf. \cite{CW92,TW05}), which only deal with the case $V=W$,
assume that $a(\cdot,\cdot)$ has a decomposition 
$
a(\cdot,\cdot)=a_0(\cdot,\cdot) + a_1(\cdot,\cdot),
$
where $a_0(\cdot,\cdot)$ is assumed to be symmetric and strongly coercive 
(i.e., it is SPD) and $a_1(\cdot,\cdot)$ is a perturbation 
of $a_0(\cdot,\cdot)$.  In this setting $a_0(\cdot,\cdot)$ 
then induces an equivalent (to $\|\cdot\|_{V}$) norm 
$\sqrt{a_0(v,v)}$ and one then works with this norm as in the SPD case. 
Unfortunately and understandably, such a setting requires that 
$a_1(\cdot,\cdot)$ is a {\em small} perturbation of $a_0(\cdot,\cdot)$, 
which is why the existing nonsymmetric 
and indefinite Schwarz framework and theory only apply to ``nearly" 
SPD problems.  Hence, it leaves more interesting and more difficult 
nonsymmetric and indefinite problems unresolved.

\subsection{Main assumptions and main idea}\label{sec-3.2}
To develop a new Schwarz framework and theory for general nonsymmetric
and indefinite problems, our only assumptions on the discrete problem
\eqref{eq2.10} are those stated in the well-posedness  
Theorem \ref{discrete_babuska}. We now restate 
those assumptions on the discrete bilinear form $a(\cdot,\cdot)$ 
and its adjoint $a^*(\cdot,\cdot)$ using the new function and 
space notation (i.e., after suppressing the sub-index $n$) as follows:

\medskip
\noindent
{\bf Main Assumptions}
\smallskip

\begin{itemize}
\item[(MA$_1$)] {\em Continuity} There exists a positive constant $C_a$ such that
\begin{align}\label{eq3.1}
|a(v,w)| \leq C_a  \|v\|_{V} \|w\|_{W} \qquad\forall v\in V, \, w\in W.
\end{align}

\item[(MA$_2$)] {\em Weak coercivity} There exists positive constants $\gamma_a$, $\beta_a$ such that
\begin{alignat}{2}\label{eq3.2}
\sup_{w\in W} \frac{a(v,w)}{\|w\|_{W}}
&\geq \gamma_a \|v\|_{V} &&\qquad\forall v\in V,\\
\sup_{v\in V} \frac{a(v,w)}{\|v\|_{V}}
&\geq \beta_a \|w\|_{W} &&\qquad\forall w\in W. \label{eq3.3}
\end{alignat}
\end{itemize}

\begin{remark}\label{rem3.1}
(a) Since $a^*(w,v)=a(v,w)$, then the continuity condition \eqref{eq3.1} 
is equivalent to 
\begin{align}\label{eq3.4}
|a^*(w,v)| \leq C_a \|w\|_{W} \|v\|_{V} \qquad\forall w\in W,\, v\in V,
\end{align}
and \eqref{eq3.3} is equivalent to
\begin{align}\label{eq3.5}
\sup_{v\in V} \frac{a^*(w,v)}{\|v\|_{V}}
\geq \beta_a \|w\|_{W} \qquad\forall w\in W. 
\end{align}

(b) Assumptions (MA$_1$) and (MA$_2$) impose some 
restrictions on the underlying Galerkin method \eqref{eq2.10}. But as we 
noted in Remark \ref{rem2.2}, these are some ``minimum"  
conditions for the Galerkin method to be practically useful. 
From that point of view, (MA$_1$) and (MA$_2$) are not restrictions at all.
\end{remark}

As it was pointed out in the previous subsection, for a general 
nonsymmetric and indefinite problem, since the discrete bilinear 
form $a(\cdot,\cdot)$ is not strongly coercive, then $a(v,v)$ is not  
a norm anymore.  In fact, $a(v,v)$ may not even be defined if $V\neq W$.
So a crucial question is what norms (if any) would $a(\cdot,\cdot)$  
induce on $V$ and $W$ which are equivalent to $\|\cdot\|_V$ 
and $\|\cdot\|_W$. It turns out that  
$a(\cdot,\cdot)$ does induce equivalent norms on both $V$ and $W$,
and these norms are hidden in the weak coercivity conditions 
\eqref{eq3.2} and \eqref{eq3.3}! This key observation
leads to the main idea of this paper; that is, we define 
the following weak coercivity induced norms:
\begin{alignat}{2}\label{eq3.6}
\|v\|_a &:=\sup_{w\in W} \frac{a(v,w)}{\normW{w}} &&\qquad\forall v\in V,\\ 
\|w\|_{a^*} &:=\sup_{v\in V} \frac{a^*(w,v)}{\normV{v}} &&\qquad\forall w\in W. 
\label{eq3.7}
\end{alignat}

Assumptions (MA$_1$) and (MA$_2$) immediately infer the following
norm equivalence result. Since its proof is trivial, we omit it.

\begin{lemma}\label{lem3.1}
The following inequalities hold:
\begin{alignat}{2}\label{eq3.8}
\gamma_a \normV{v} &\leq \|v\|_a \leq C_a \normV{v}  &&\qquad\forall v\in V,\\ 
\beta_a \normW{w} &\leq \|w\|_{a^*} \leq C_a \normW{w} &&\qquad\forall w\in W.
\label{eq3.9}
\end{alignat}
\end{lemma}

We conclude this subsection by noting that the variational 
setting laid out so far is a Banach space setting. No Hilbert
space structure is required for the space $V$ and $W$.  This is
not only mathematically interesting but also practically valuable 
because for some PDE application problems it is 
imperative to work in a Banach space setting. We also note that
if $V=W$ and $a(\cdot,\cdot)$ is SPD (i.e., it is symmetric and
strongly coercive), then $\|v\|_a= \|v\|_{a^*} = \sqrt{a(v,v)}$.
Hence, we recover the standard bilinear form induced norm!

\subsection{Space decomposition and local solvers} \label{sec-3.3}
It is well known \cite{DW90,Xu92,SBG96,XZ02,TW05} that Schwarz 
domain decomposition methods can be presented abstractly in the 
framework of the space decomposition method.  In particular, the physical domain 
decomposition provides a practical and effective way to construct 
the required space decomposition and local solvers in the method.
To some extent, the space decomposition method to the Schwarz
domain decomposition method is what the LU factorization is to the 
classical Gaussian elimination method.

Like in the SPD Schwarz framework (cf. \cite{TW05}), there are 
two essential ingredients in our nonsymmetric and 
indefinite Schwarz framework, namely, (i) {\em construction of a 
pair of ``compatible" space decompositions for $V$ and $W$} and 
(ii) {\em construction of a local solver (or local discrete bilinear 
form) on each pair of local spaces}. However, there is an obvious and 
crucial difference between the SPD Schwarz framework and our 
nonsymmetric and indefinite Schwarz framework.  When
$V\neq W$, our framework requires space decompositions 
for both spaces $V$ and $W$, and these 
two space decompositions must be chosen compatibly in the 
sense to be described below.
 
Let 
\[
V_j\subset X,\qquad W_j\subset X \qquad\mbox{for } j=0,1,2,\cdots, J,
\] 
be two sets of reflexive Banach spaces with norms 
$\normVj{\cdot}$ and $\normWj{\cdot}$ respectively.  We note that 
$V_0$ and $W_0$ are used to denote the so-called {\em coarse spaces} 
in the domain decomposition context.  For $j=0,1,2,\cdots, J$, let 
\[
\cR_j^\dag: V_j\to V, \qquad \cS_j^\dag:W_j\to W 
\]
denote some {\em prolongation operators}. 

\begin{remark}\label{rem3.3}
In the Schwarz method literature (cf. \cite{TW05,SBG96,Xu92}), 
$R_j^T$ is often used to denote both the prolongation operator
from $V_j$ to $V$ and its matrix representation.  Such a choice 
of notation is due to the fact that the matrix representation of
the {\em not-explicitly-defined} restriction operator $R_j$ from
$V$ to $V_j$ is always chosen to be the transpose of the matrix 
representation of the prolongation operator. As expected, such a 
dual role notation may be confusing to some readers. To avoid such 
a potential confusion we use different notations for operators and 
their matrix representations throughout this paper.  

We also like to note that in the construction 
of all Schwarz methods the restriction operators/matrices  
are not ``primary" operators/matrices but ``derivative" 
operators/matrices in the sense that they are not chosen independently. 
Instead, they are determined by the prolongation 
operators/matrices. One often first defines the 
matrix representation of the (desired) restriction 
operator as the transpose of the the matrix representation 
of the prolongation operator and then defines the restriction
operator to be the unique linear operator which has the chosen 
matrix representation (under the same bases in which the prolongation 
matrix is obtained). This will also be the approach 
adopted in this paper for defining our restriction operators
(see Definition \ref{def3.2}).  Clearly, such a definition of
the restriction operators is not only abstract but 
also depends on the choices of the bases of the underlying 
function spaces. However, its simplicity and convenience at 
the matrix level, which are what really matter in practice,
make the definition appealing and favorable so far.
\end{remark}

Suppose that the following relations hold:
\begin{alignat}{2}\label{eq3.10}
\cR_j^\dag V_j &\subsetneq V,  &\qquad \cS_j^\dag W_j &\subsetneq W  
\quad\mbox{for } j=0,1,2,\cdots, J,\\
V &= \sum_{j=0}^J \cR_j^\dag V_j, &\quad 
W &= \sum_{j=0}^J \cS_j^\dag W_j, \label{eq3.11}
\end{alignat}
where $\cR_j^\dag V_j$ and $\cS_j^\dag W_j$ stand for the ranges of the 
linear operators $\cR_j^\dag$ and $\cS_j^\dag$ respectively.

Associated with each pair of local spaces $(V_j, W_j)$ for $j=0,1,2\cdots,J$, 
we introduce a local discrete bilinear form $a_j(\cdot, \cdot)$ defined 
on $V_j\times W_j$, which can be taken either as the restriction of 
global discrete bilinear form $a(\cdot, \cdot)$ on $V_j\times W_j$
or as some approximation of the restriction of $a(\cdot, \cdot)$ 
on $V_j\times W_j$. We call these two choices of local discrete bilinear 
form $a_j(\cdot, \cdot)$ an {\em exact local solver} and
an {\em inexact local solver} respectively. After the local discrete
bilinear forms are chosen, we can define what constitutes as
a {\em compatible space decomposition}. 

\begin{definition}\label{def3.1}
(i) A pair of spaces $V_j$ and $W_j$ are said to be compatible with
respect to $a_j(\cdot,\cdot)$ if they satisfy the following conditions:

\smallskip
\begin{itemize}
\item[{\rm (LA$_1$)}] {\em Local continuity}. There exists a positive 
constant $C_{a_j}$ such that
\begin{align}\label{eq3.14}
|a_j(v,w)| \leq C_{a_j} \normVj{v} \normWj{w} \qquad\forall v\in V_j, 
\, w\in W_j.
\end{align}

\item[{\rm (LA$_2$)}] {\em Local weak coercivity}. There exist 
positive constants $\gamma_{a_j}$ and $\beta_{a_j}$ such that
\begin{alignat}{2}\label{eq3.15}
\sup_{w\in W_j} \frac{a_j(v,w)}{\normWj{w}}
&\geq \gamma_{a_j} \normVj{v} &&\qquad\forall v\in V_j,\\
\sup_{v\in V_j} \frac{a_j(v,w)}{\normVj{v}}
&\geq \beta_{a_j} \normWj{w} &&\qquad\forall w\in W_j. \label{eq3.16}
\end{alignat}
\end{itemize}

(ii) A pair of space decompositions $\{V_j\}_{j=0}^J$ and $\{W_j\}_{j=0}^J$
of $V$ and $W$ satisfying \eqref{eq3.10}--\eqref{eq3.11}
are said to be compatible if each pair of $V_j$ and $W_j$ 
is compatible with respect to $a_j(\cdot,\cdot)$ 
for $j=0,1,2,\cdots,J$.
\end{definition}

Obviously conditions (LA$_1$) and (LA$_2$) on $a_j(\cdot, \cdot)$
are the analogies of (MA$_1$) and (MA$_2$) on $a(\cdot, \cdot)$. 
By Theorem \ref{discrete_babuska}, these conditions guarantee 
that the local problem of seeking $u_j\in V_j$ such that
\begin{equation}\label{eq3.17}
a_j(u_j,w_j)=f_j(w_j) \qquad \qquad \forall \ w_j \in W_j
\end{equation}
is uniquely solvable for any given bounded linear functional 
$f_j$ on $W_j$. Moreover, (LA$_1$) and (LA$_2$) are ``minimum" conditions for achieving 
such a guaranteed unique solvability (cf. Remark \ref{rem2.2}).
Furthermore, like its global counterpart, the local weak coercivity 
condition (LA$_2$) induces the following two equivalent 
norms on $V_j$ and $W_j$: 
\begin{alignat}{2}\label{eq3.18}
\|v\|_{a_j} &:=\sup_{w\in W_j} \frac{a_j(v,w)}{\normWj{w}} 
&&\qquad\forall v\in V_j,\\ 
\|w\|_{a_j^*} &:=\sup_{v\in V_j} \frac{a_j^*(w,v)}{\normVj{v}} 
&&\qquad\forall w\in W_j, \label{eq3.19}
\end{alignat}
where $a_j^*(w,v):=a_j(v,w)$ for any $(v,w)\in V_j\times W_j$.

Trivially, we have 

\begin{lemma}\label{lem3.2}
Suppose that $V_j$ and $W_j$ are compatible with respect to
$a_j(\cdot, \cdot)$. Then the following inequalities hold:
\begin{alignat}{2}\label{eq3.20}
\gamma_{a_j} \normVj{v} &\leq \|v\|_{a_j} 
\leq C_{a_j} \normVj{v}  &&\qquad\forall v\in V_j,\\ 
\beta_{a_j} \normWj{w} &\leq \|w\|_{a_j^*} 
\leq C_{a_j} \normWj{w} &&\qquad\forall w\in W_j.  \label{eq3.21}
\end{alignat}
\end{lemma}

\subsection{Additive Schwarz method} \label{sec-3.4}
Throughout this section, we assume that we 
are given a global discrete problem \eqref{eq2.10}, and 
the global discrete bilinear form $a(\cdot, \cdot)$ fulfills
the main assumptions (MA$_1$) and (MA$_2$) so that problem \eqref{eq2.10}
has a unique solution $u\in V$. In addition, we assume we are given 
a pair of space decompositions $\{V_j\}_{j=0}^J$ and $\{W_j\}_{j=0}^J$
of $V$ and $W$, the prolongation operators
$\{\cR_j^\dag\}_{j=0}^J$ and $\{\cS_j^\dag\}_{j=0}^J$, 
and the local discrete bilinear forms 
$\{a_j(\cdot, \cdot)\}_{j=0}^J$ such that the given space 
decompositions are compatible with respect to the given local 
discrete bilinear forms in the sense of Definition \ref{def3.1}. 
Our goal in this subsection is to construct 
the additive Schwarz method for problem \eqref{eq2.10} using the 
given information. 

To continue, we now introduce two sets of {\em projection-like} 
operators $\ctP_j: V\to V_j$ and $\ctQ_j: W\to W_j$ for $j=0,1,2,
\cdots,J$.  These projection-like-operators will serve as the building blocks for the 
constructions of both our additive and multiplicative 
Schwarz methods.  For any fixed $v\in V$ and $w\in W$, 
define $\ctP_j v\in V_j$ and $\ctQ_j w\in W_j$ by
\begin{alignat}{2}\label{eq3.22}
a_j\bigl(\ctP_j v,w_j\bigr) &:= a\bigl(v, \cS_j^\dag w_j\bigr)
&&\qquad\forall w_j\in W_j,\\ 
a_j^*\bigl(\ctQ_j w,v_j\bigr) &:= a^*\bigl(w, \cR_j^\dag v_j\bigr)
&&\qquad\forall v_j\in V_j. \label{eq3.23}
\end{alignat}
We recall that $a_j^*(w_j,v_j)=a_j(v_j,w_j)$ for all $v_j\in V_j$ 
and $w_j\in W_j$. We also note that since $V_j$ and $W_j$ are 
assumed to be compatible, Theorem \ref{discrete_babuska} then ensures 
both $\ctP_j$ and $\ctQ_j$ are well defined for $j=0,1,\cdots, J$. 

Since $V_j$ and $W_j$ may not be subspaces of $V$ and $W$, 
$\ctP_j v$ and $\ctQ_j w$ may not belong to $V$ and $W$.
To pull them back to the global discrete spaces $V$ and $W$, we 
appeal to the {\em prolongation operators} $\cR_j^\dag$ and $\cS_j^\dag$
for help. Define the composite operators 
\begin{align}\label{eq3.24}
\cP_j:=\cR_j^\dag\circ \ctP_j,\qquad \cQ_j:=\cS_j^\dag\circ \ctQ_j 
\qquad\mbox{for } j=0,1,2, \cdots,J.
\end{align}
Trivially, we have $\cP_j: V\to V$ and $\cQ_j: W\to W$ for $j=0,1,2, \cdots,J$.

We now are ready to define the following additive Schwarz operators. 
Following \cite{DW90,Xu92,SBG96,TW05} we define
\begin{align}\label{eq3.25}
\Pad &:= \cP_0+\cP_1+\cP_2+\cdots + \cP_J,\\
\Qad &:= \cQ_0+\cQ_1+\cQ_2 +\cdots +\cQ_J. \label{eq3.26}
\end{align}

The matrix interpretation of the additive operator $\Pad$ 
is similar to but slightly more complicated than the one in the 
SPD Schwarz framework. In particular, 
the additive operator $\Qad$ does not exist in the 
the SPD framework. For the reader's convenience, we give below a 
brief matrix interpretation for both $\Pad$ and $\Qad$.

Fixing a basis for each of $V, W, V_j$ and $W_j$, let $A$ and $A_j$ denote 
respectively the global and local stiffness matrices of the 
bilinear forms $a(\cdot,\cdot)$ and $a_j(\cdot,\cdot)$ with respect to
the given bases. Let $R_j^\dag, S_j^\dag, \tP_j, \tQ_j, P_j, Q_j,
P_{\mbox{\tiny ad}}$ and $Q_{\mbox{\tiny ad}}$ denote the matrix 
representations of 
the linear operators $\cR_j^\dag,\cS_j^\dag,\ctP_j,\ctQ_j,
\cP_j,\cQ_j, \Pad$ and $\Qad$ with respect to the given bases. Lastly, 
let $A^T, A_j^T, R_j^{\dag T}$ and $S_j^{\dag T}$ denote 
the matrix transposes of $A, A_j, R_j^\dag$ and $S_j^\dag$.    

Using the above notation and the well-known fact that composite linear 
operators are represented by matrix multiplications, we obtain from
\eqref{eq3.22} and \eqref{eq3.23} that
\begin{alignat}{2}\label{eq3.27}
A_j \tP_j \bv &:= S_j^{\dag T} A\bv &&\qquad\forall \bv\in \bR^n,\\ 
A_j^T \tQ_j\bw &:= R_j^{\dag T} A^T\bw &&\qquad\forall \bw\in \bR^n. 
\label{eq3.28}
\end{alignat}
Thus,
\begin{alignat}{2}\label{eq3.29}
\tP_j &= A_j^{-1} S_j^{\dag T} A, &\qquad 
P_j &= R_j^\dag A_j^{-1} S_j^{\dag T} A, \\
\tQ_j &= A_j^{-T} R_j^{\dag T} A^T, &\qquad 
Q_j &= S_j^\dag A_j^{-T} R_j^{\dag T} A^T, \label{eq3.30}
\end{alignat}
where $A_j^{-1}$ and $A_j^{-T}$ denote the inverse matrices of $A_j$ and 
$A_j^T$, respectively. We also note that the compatibility 
assumptions (LA$_1$) and (LA$_2$) imply that $A_j^{-1}$ and $A_j^{-T}$
do exist.

Finally, it follows from \eqref{eq3.25}, \eqref{eq3.26}, \eqref{eq3.29} and
\eqref{eq3.30} that 
\begin{align}\label{eq3.31}
P_{\mbox{\tiny ad}} & =  R_0^\dag A_0^{-1} S_0^{\dag T} A
+\sum_{j=1}^J  R_j^\dag A_j^{-1} S_j^{\dag T} A
\\
Q_{\mbox{\tiny ad}} & = S_0^\dag A_0^{-T} R_0^{\dag T} A^T
+\sum_{j=1}^J S_j^\dag A_j^{-T} R_j^{\dag T} A^T.
\label{eq3.32}
\end{align}

From the above expressions we obtain the following two additive Schwarz 
preconditioners for both $A$ and its transpose $A^T$:
\begin{align}\label{eq3.33}
B &:= R_0^\dag A_0^{-1} S_0^{\dag T}
+\sum_{j=1}^J  R_j^\dag A_j^{-1} S_j^{\dag T},\\
\quad B^\dag &:= S_0^\dag A_0^{-T} R_0^{\dag T}
+\sum_{j=1}^J S_j^\dag A_j^{-T} R_j^{\dag T}. \label{eq3.33a}
\end{align}
It is interesting to note that $B^\dag= B^T$ which means 
that the nonsymmetric Schwarz preconditioner $B$ can be used 
to precondition both the linear system \eqref{eq2.10} and its 
adjoint system without any additional cost. 

As it was already alluded to in Remark \ref{rem3.3}, we now formally define
our restriction operators $\{\cR_j\}$ and $\{\cS_j\}$.

\begin{definition}\label{def3.2}
For $j=0,1,2,\cdots,J$, let $\cR_j: V\to V_j$ 
(resp. $\cS_j: W\to W_j$) be the unique linear operator 
whose matrix representation is given by $S_j^{\dag T}$ 
(resp. $R_j^{\dag T}$) under the same bases of $V,W,V_j$ and $W_j$ 
in which $R_j^{\dag T}$ and $S_j^{\dag T}$ are obtained.
\end{definition}

By the design, the matrix representations $R_j$ and $S_j$
of $\cR_j$ and $\cS_j$ satisfy $R_j=S_j^{\dag T}$ and 
$S_j=R_j^{\dag T}$.

\subsection{Multiplicative Schwarz method} \label{sec-3.5}
The multiplicative Schwarz methods for solving problem 
\eqref{eq2.10} refer to various generalizations
of the original Schwarz alternating iterative method 
(cf. \cite{BPWX91,Xu92}). However, they also can be 
formulated as linear iterations on some preconditioned 
systems (cf. \cite{TW05}). In this paper we adopt the 
latter point of view to present our nonsymmetric and 
indefinite multiplicative Schwarz methods. We shall 
use the same notation as in Section \ref{sec-3.4}.

We first introduce the following two so-called {\em error propagation 
operators}:
\begin{align}\label{eq3.34}
\Emu &:=(\cI-\cP_J)\circ(\cI-\cP_{J-1})\circ \cdots \circ (\cI-\cP_0),\\
\Esy &:=(\cI-\cP_0)\circ (\cI-\cP_1)\circ\cdots\circ (\cI-\cP_J)\circ
(\cI-\cP_{J})\circ \cdots \circ (\cI-\cP_0). \label{3.35} 
\end{align}
where $\cI$ denotes the identity operator on $V$ or on $W$.
We then define the following two ``preconditioned" operators:
\begin{align}\label{eq3.36}
\Pmu :=\cI-\Emu, \qquad\quad
\Psy:=\cI-\Esy.
\end{align}

It is easy to check that the algebraic matrix representations
of the above operators are, respectively, 
\begin{align}\label{eq3.37}
E_{\mbox{\tiny mu}} &:=(I-P_J)(I-P_{J-1})\cdots (I-P_0),\\
E_{\mbox{\tiny sy}} &:=(I-P_0)(I-P_1)\cdots (I-P_J) 
(I-P_{J})\cdots(I-P_1) (I-P_0), \label{3.38}\\
P_{\mbox{\tiny mu}} &:=I-E_{\mbox{\tiny mu}}, \label{3.39} \\
P_{\mbox{\tiny sy}} &:=I- E_{\mbox{\tiny sy}}. \label{3.40}
\end{align}
Then our multiplicative Schwarz iterative methods are defined as 
\begin{align}\label{eq3.41}
\bu^{(k+1)} = (I-C) \bu^{(k)} +\bg 
=E \bu^{(k)} +\bg, \qquad k\geq 0
\end{align}
where $(C,E)$ are either $(P_{\mbox{\tiny mu}}, E_{\mbox{\tiny mu}})$ 
or $(P_{\mbox{\tiny sy}}, E_{\mbox{\tiny sy}})$,
and $\bg$ takes either $\bg_{\mbox{\tiny mu}}\in \bR^n$ or 
$\bg_{\mbox{\tiny sy}}\in \bR^n$ which are easily computable 
from $\bbf$ in \eqref{eq2.11}.

\begin{remark}\label{rem3.2}
(a) Clearly, the case with the triple $(P_{\mbox{\tiny mu}}, E_{\mbox{\tiny mu}},
\bg_{\mbox{\tiny mu}})$ corresponds to the classical multiplicative 
Schwarz method for \eqref{eq2.11} (cf. \cite{BPWX91}).

(b) The case with the triple $(P_{\mbox{\tiny sy}}, E_{\mbox{\tiny sy}}, 
\bg_{\mbox{\tiny sy}})$ can be regarded as a ``symmetrized" 
multiplicative Schwarz method for nonsymmetric and indefinite 
problems. However, we note that the operator 
$\cE_{\mbox{\tiny sy}}$ and matrix $E_{\mbox{\tiny sy}}$ are not 
symmetric in general because $\{\cP_j\}$ and $\{P_j\}$ may not be
symmetric.


(c) Unlike in the SPD case, the norm $\|\Emu\|_a$ could be larger than $1$
for convection-dominant problems as shown by the numerical tests given 
in Section \ref{sec-5} although the multiplicative Schwarz method 
appears to be convergent in all those tests. Consequently, the 
convergent behavior of the multiplicative Schwarz method presented  
above is more complicated than its SPD counterpart.

\end{remark}

\subsection{A hybrid Schwarz method} \label{sec-3.6}
In this subsection, we consider a hybrid Schwarz method which combines the  
additive Schwarz idea (between subdomains) and the multiplicative Schwarz
idea (between levels). The hybrid method is expected to take advantage
of both additive and multiplicative Schwarz methods.

The iteration operator of our hybrid Schwarz method is given by
\begin{eqnarray}\label{hy1}
\cE_{\mbox{\tiny hy}}&:=(\cI-\alpha \cP_0)(\cI- \widehat{\cP} ), 
\qquad\mbox{where}\quad \widehat{\cP}:=\sum_{j=1}^J \cP_j. \\
\cG_{\mbox{\tiny hy}}&:=(\cI-\alpha \cQ_0)(\cI- \widehat{\cQ} ), 
\qquad\mbox{where}\quad \widehat{\cQ}:=\sum_{j=1}^J \cQ_j. \label{hyb2}
\end{eqnarray}
Thus, the ``preconditioned"  hybrid Schwarz operator has the following form:
\begin{eqnarray}\label{hyb3}
\cP_{\mbox{\tiny hy}}&:=\cI-\cE_{\mbox{\tiny hy}}=\alpha \cP_0 
+ (\cI-\alpha \cP_0) \widehat{\cP}.
\\
\cQ_{\mbox{\tiny hy}}&:=\cI-\cG_{\mbox{\tiny hy}}=\alpha \cQ_0 
+ (\cI-\alpha \cQ_0) \widehat{\cQ}, \label{hyb4}
\end{eqnarray}
where $\alpha$, called a relaxation parameter,  is an undetermined
positive constant.

Since the corresponding matrix representations of $\cE_{\mbox{\tiny hy}},
\cP_{\mbox{\tiny hy}}, \cG_{\mbox{\tiny hy}}$, and $\cQ_{\mbox{\tiny hy}}$
are easy to write down, we omit them to save space.

\section{An abstract Schwarz convergence theory for nonsymmetric and
indefinite problems} \label{sec-4}
In this section we shall first establish condition number estimates for 
additive Schwarz operator $\Pad$ and for its matrix
representation $P_{\mbox{\tiny ad}}$.
We then present a condition number estimate for the hybrid 
operator $\cP_{\mbox{\tiny hy}}$.

\subsection{Structure assumptions}\label{sec-4.1}
Our convergence theory rests on the following Structure Assumptions:

\begin{itemize}
\item[{\rm (SA$_0$)}] {\em Compatibility assumption}. Assume that 
$\{(V_j,W_j)\}_{j=0}^J$ is a compatible decomposition of $(V, W)$
in the sense of Definition \ref{def3.1}. 

\item[{\rm (SA$_1$)}] {\em Energy stable decomposition assumption}. 
There exist positive constants $\CV$ and $\CW$ such 
that every pair $(v,w)\in V\times W$ admits a decomposition 
\[
v=\sum_{j=0}^J \cR_j^\dag v_j,\qquad\quad w=\sum_{j=0}^J \cS_j^\dag w_j,
\]
with $v_j\in V_j$ and $w_j\in W_j$ such that 
\begin{align}\label{eq4.1x}
\sum_{j=0}^J \|v_j\|_{a_j} &\leq \CV \|v\|_a, \\
\sum_{j=0}^J \normWj{w_j} &\leq \CW \normW{w}.  \label{eq4.1y}
\end{align}

\item[{\rm (SA$_2$)}] {\em Strengthened generalized Cauchy-Schwarz inequality 
assumption}. There exist constants $\theta_{ij}\in [0,1]$ for 
$i,j=0,1,2,\cdots,J$ such that 
\begin{align}\label{eq4.2}
a(\cR_i^\dag v_i, \cS_j^\dag w_j) &\leq \theta_{ij} \|\cR_i^\dag v_i\|_a  
\normW{\cS_j^\dag w_j} \qquad\forall v_i\in V_i,\, w_j\in W_j.
\end{align}

\item[{\rm (SA$_3$)}] {\em Local stability assumption}. There exist 
positive constants $\omeV$ and $\omeW$ such that for $j=0,1,2,\cdots, J$
\begin{alignat}{2}\label{eq4.3x}
\|\cR_j^\dag v_j\|_a &\leq \omeV \|v_j\|_{a_j}  &&\qquad \forall v_j\in V_j,\\
\normW{\cS_j^\dag w_j} &\leq \omeW \normWj{w_j} &&\qquad \forall w_j\in W_j.
\label{eq4.3y}
\end{alignat}

\item[{\rm (SA$_4$)}] {\em Approximability assumption}. There exist 
(small) positive constants $\deltaV, \widehat{\delta}_V,  \deltaW$ and 
$\widehat{\delta}_W$ such that for $i=0,1,2,\cdots,J$ and $j = 1,2, \cdots, J$
\begin{alignat}{2}\label{eq4.4x}
\|\ctP_j \bigl( v-\cP_i v\bigr) \|_{a_j}
&\leq \theta_{ij} \deltaV  \|v\|_a  &&\qquad\forall v\in V, \\
\| \ctP_0 \bigl( v - \cP_0 v \bigr) \|_{a_0} & \leq \deltaV \| v \|_a && \qquad \forall v \in V, \\
 \| \ctP_0 \bigl( v - \widehat{\cP} v \bigr) \|_{a_0}
&\leq \widehat{\delta}_V \| v \|_a && \qquad \forall v \in V, \\
\normWj{\ctQ_j \bigl( w-\cQ_i w\bigr)}
&\leq \theta_{ij} \deltaW  \normW{w}  &&\qquad\forall w\in W, \label{eq4.4y} \\
\normWnot{\ctQ_0 \bigl( w - \cQ_0 w \bigr)} & \leq \deltaW \normW{ w } 
&& \qquad \forall w \in W, \\
\normWnot{\ctQ_0 \bigl( w - \widehat{\cQ} w \bigr)} 
& \leq \widehat{\delta}_W \normW{w} && \qquad \forall w \in W,
%
\end{alignat}
where $\theta_{ij}$ are the same as in (SA$_2$).
$\widehat{\cP} := \sum_{i = 1}^{J} \cP_i$ and
$\widehat{\cQ} := \sum_{i = 1}^{J} \cQ_i$.
\end{itemize}

We now explain the rationale and motivation of each assumption listed above.

\begin{remark}\label{rem4.1}
(a) We note that $\|\cdot\|_a$ and $\|\cdot\|_{a^*}$ are defined in
\eqref{eq3.6} and \eqref{eq3.7}, and $\|\cdot\|_{a_j}$ and $\|\cdot\|_{a_j^*}$
are defined in \eqref{eq3.18} and \eqref{eq3.19}.

(b) In the SPD Schwarz theory (cf. \cite{TW05}),  the local bilinear forms 
are always assumed to be strongly coercive which then infers the invertibilty 
of the local stiffness matrices. However, such an assumption is often 
not listed explicitly. On the other hand, in our nonsymmetric and indefinite
Schwarz theory, the invertibilty of the local stiffness matrices may
not hold.  We explicitly list it as an assumption in (SA$_0$).

(c) For a given compatible pair of space decompositions $\{(V_j,W_j)\}_{j=0}^J$, 
decompositions of each function $v\in V$ and $w\in W$ may not be unique.
Assumption (SA$_1$) assumes that there exists at least one decomposition
which is energy stable for every function in $V$ and $W$. It
imposes a constraint on both the choice of the space decompositions
$\{(V_j,W_j)\}_{j=0}^J$ and on the choice of the local bilinear 
forms $\{a_j(\cdot,\cdot)\}_{j=0}^J$.

(d) We note that different norms are used for two functions 
on the right-hand side of \eqref{eq4.2}, and $\theta_{ij}$ is 
defined for $i,j=0,1,2,\cdots,J$. We set $\Theta=[\theta_{ij}]_{i,j=0}^J$  
and note that $\Theta$ is a $(J+1)\times (J+1)$ matrix.
We shall also use the submatrix $\widehat{\Theta}:=\Theta(1:J,1:J)$ 
in our convergence analysis to be given in Section \ref{sec-4}.
Since the bilinear form $a(\cdot,\cdot)$ is not an inner product, the 
standard Cauchy-Schwarz inequality does not hold in general. But it does hold 
in this generalized sense with $\theta_{ij}=1$, see Lemma \ref{lem4.1}. 
Moreover, we expect that each pair $(V_j,W_j)$ for $1\leq i,j\leq J$ only 
interacts with very few remaining pairs in the space
decomposition $\{(V_j,W_j)\}_{j=1}^J$.  Hence, the matrix
$\widehat{\Theta}$, which is symmetric, is expected to be 
sparse and nearly diagonal in most applications.
On the other hand, we expect that $\theta_{0j}=\theta_{i0}=1$ for
$i,j=1,2,\cdots, J$. 

(e) Local stability assumption (SA$_3$) imposes a condition on the choice 
of the prolongation operators $\cR_j^\dag$ and $\cS_j^\dag$. 
It requires that these operators are bounded operators.

(f) Assumption (SA$_4$), which does not appear in the SPD theory, 
imposes a local approximation condition on the projection-like 
operators $\{\ctP_j\}$ and $\{\ctQ_j\}$ and on the restriction 
operators $\{\cR_j\}$ and $\{\cS_j\}$. 


(g) Because of the norm equivalence properties \eqref{eq3.8}, \eqref{eq3.9},
\eqref{eq3.20} and \eqref{eq3.21}, one can easily replace the 
weak coercivity induced norms by their equivalent underlying space norms 
or vice versa in all assumptions (SA$_1$)--(SA$_4$). However, one 
must track all the constants resulting from the changes. The main 
reason for using the current forms of the assumptions is that they 
allow us to give a cleaner presentation of our nonsymmetric
and indefinite Schwarz convergence theory to be described below.
\end{remark}

\subsection{Condition number estimate for $\Pad$}\label{sec-4.2}

First, we state the following simple lemma.

\begin{lemma}\label{lem4.1}
The following generalized Cauchy-Schwarz inequalities hold:
\begin{alignat}{2}\label{eq4.5a}
a(v, w) &\leq \|v\|_a \normW{w}  &&\qquad\forall v\in V,\, w\in W,\\
a(v, w) &\leq \normV{v} \|w\|_{a^*}  &&\qquad\forall v\in V,\, w\in W, \label{eq4.5b}\\
a_j(v_j, w_j) &\leq \|v_j\|_{a_j} \normWj{w} &&\qquad\forall v_j\in V_j,\, 
w_j\in W_j,\,\, j=0,1,\cdots, J, \label{eq4.5c}\\
a_j(v_j, w_j) &\leq \normVj{v} \|w_j\|_{a_j^*}  &&\qquad\forall 
v_j\in V_j,\, w_j\in W_j, \,\, j=0,1,\cdots, J. \label{eq4.5d}
\end{alignat}
\end{lemma}

\begin{proof}
\eqref{eq4.5a}--\eqref{eq4.5d} are immediate consequences of the 
definitions of the norms $\|\cdot\|_a, \|\cdot\|_{a^*}$, $\|\cdot\|_{a_j}$
and $\|\cdot\|_{a_j^*}$. 
\end{proof}

\begin{lemma}\label{lem4.2}
Under assumptions (SA$_0$) and (SA$_3$), the following estimates hold: 
\begin{alignat}{2}\label{eq4.10a}
\|\ctP_j v\|_{a_j} &\leq \omeW \|v\|_a  &&\qquad \forall v\in V,\quad
j=0,1,\cdots,J. \\
\|\cP_j v\|_a &\leq \omeV \omeW \|v\|_a  &&\qquad \forall v\in V,\quad
j=0,1,\cdots,J. \label{eq4.10} \\
\|\ctQ_j w\|_{a^*_j} &\leq \omeV C_{a_j} \beta_a^{-1} \|w\|_{a^*}  
&&\qquad \forall w\in W,\quad j=0,1,\cdots,J. \label{eq4.10c}\\
\|\cQ_j w\|_{a^*} &\leq \omeV \omeW C_a C_{a_j} \beta_a^{-1} \beta^{-1}_{a_j} \|w\|_{a^*}  &&\qquad \forall w\in W,\quad
j=0,1,\cdots,J. \label{eq4.10d} \\
\normV{\cP_j v} &\leq \omeV \omeW C_a \gamma_a^{-1} \normV{v}  
&&\qquad \forall v\in V,\quad j=0,1,\cdots,J. \label{eq4.10e} \\
\normW{\cQ_j w} &\leq \omeV \omeW C_{a_j} \beta^{-1}_{a_j} 
\normW{w} &&\qquad \forall w\in W,\quad j=0,1,\cdots,J. \label{eq4.10f}
\end{alignat}
\end{lemma}

\begin{proof}
For any $v\in V$, by assumption (SA$_3$) and Lemma \ref{lem4.1} we get
for $j=0,1,\cdots,J$,
\begin{alignat}{2}\label{eq4.12}
\|\ctP_j v\|_{a_j}
&= \sup_{ w_j\in W_j} \frac{ a_j(\ctP_j v, w_j)}{\normWj{w_j} }  && \\
&= \sup_{ w_j\in W_j} \frac{ a(v, \cS_j^\dag w_j)}{\normWj{w_j} }
&&\qquad (\mbox{by \eqref{eq3.22}})
\nonumber \\
&\leq \sup_{ w_j\in W_j} \frac{\|v\|_a\,\normW{\cS_j^\dag w_j}}{\normWj{w_j}}  
&&\qquad (\mbox{by \eqref{eq4.5a}})
\nonumber \\
&\leq \omeW \|v\|_a. &&\qquad (\mbox{by \eqref{eq4.3y}}) \nonumber
\end{alignat}
Hence, \eqref{eq4.10a} holds. \eqref{eq4.10} follows immediately
from \eqref{eq4.10a} and \eqref{eq4.3x}.  By assumption (SA$_3$) and Lemma \ref{lem4.1} we obtain
\begin{align*}
 \| \ctQ w \|_{a^*_j} &= \sup_{v_j \in V_j} \frac{a_j (v_j, \ctQ w)}{\normVj{v_j}}  \\
	&= \sup_{v_j \in V_j} \frac{a (\cR_j^\dag v_j, w)}{\normVj{v_j}} &\qquad (\mbox{by \eqref{eq3.23}}) \\
	& \leq \sup_{v_j \in V_j} \frac{\| \cR^\dag_j v_j \|_a \normW{w}}{\normVj{v_j}} &\qquad (\mbox{by \eqref{eq4.5a}})  \\
	& \leq \sup_{v_j \in V_j} \frac{\omeV \| v_j \|_{a_j} \normW{w}}{\normVj{v_j}} &\qquad (\mbox{by \eqref{eq4.3x}}) \\
	& \leq \omeV C_{a_j} \normW{w} &\qquad (\mbox{by \eqref{eq3.20}}) \\
	& \leq \omeV C_{a_j} \beta_a^{-1} \| w \|_{a^*}. &\qquad (\mbox{by \eqref{eq3.9}})
\end{align*}
Hence, \eqref{eq4.10c} holds.  \eqref{eq4.10d} follows from \eqref{eq4.10c}, \eqref{eq3.8}, \eqref{eq4.3y}, and \eqref{eq3.21}. From the proof for \eqref{eq4.10c} we can obtain $\| \ctQ_j w \|_{a^*_j} \leq \omeV C_{a_j} \|w \|_W$.  This result along with \eqref{eq4.3y} and \eqref{eq3.21} yields \eqref{eq4.10f}.  The proof is complete.
\end{proof}

We now are ready to give an upper bound estimate for 
the additive Schwarz operator $\Pad$.

\begin{proposition}\label{prop4.1}
Under assumptions (SA$_0$)--(SA$_3$) the following estimate holds:
\begin{align}\label{eq4.13}
\|\Pad v\|_a \leq \omeV \omeW \bigl[1+ \omeW\CW N(\Theta)\bigr] 
\|v\|_a \qquad\forall v\in V,
\end{align}
where $\Theta=[\theta_{ij}]_{i,j=0}^J$, $N(\Theta)=\max\{N_j(\Theta);\, 
0\leq j\leq J\}$ and $N_j(\Theta)$ denotes the number of nonzero entries 
in the vector $\Theta(1:J,j)$, i.e., the number of nonzeros among the 
last $J$ entries of the $j$th column of the matrix $\Theta$.

\end{proposition}

\begin{proof}
For any $w\in W$, let $\{w_j\}$ be an energy stable decomposition of $w$ as
defined in (SA$_1$). By the definition of $\Pad$,
\eqref{eq4.5a}, \eqref{eq4.2}, \eqref{eq4.10}, 
\eqref{eq4.3y}, and \eqref{eq4.1y}
we get for any $v\in V$
\begin{align}\label{eq4.15}
a(\Pad v,w) &= a(\cP_0 v, w) +\sum_{i=1}^J a(\cP_i v, w) \\
&= a(\cP_0 v, w) +\sum_{i=1}^J\sum_{j=0}^J  a(\cR_i^\dag \ctP_i v, \cS_j^\dag w_j) 
\nonumber \\
&\leq \|\cP_0 v\|_a \normW{w}
+\sum_{i=1}^J\sum_{j=0}^J \theta_{ij} \|\cP_i v\|_a \normW{\cS_j^\dag w_j} 
\nonumber \\
&\leq \omeV \omeW \|v\|_a \Bigl\{ \normW{w}+ 
\sum_{j=0}^J N_j(\Theta) \normW{\cS_j^\dag w_j} \Bigr\} 
\nonumber \\
&\leq \omeV \omeW \|v\|_a \Bigl\{ \normW{w}+
\omeW N(\Theta) \sum_{j=0}^J \normWj{w_j} \Bigr\}  \nonumber \\
&\leq \omeV \omeW \|v\|_a \Bigl\{ \normW{w}+
\omeW N(\Theta) \CW \normW{w} \Bigr\}  \nonumber \\
&= \omeV \omeW \bigl[ 1+ \omeW \CW N(\Theta) \bigr] \|v\|_a \normW{w}.  
\nonumber 
\end{align}
Hence, \eqref{eq4.13} holds. The proof is complete.
\end{proof}

As expected, it is harder to get a lower bound estimate for the 
additive Schwarz operator $\Pad$. Such a bound then readily provides
an upper bound for $\Pad^{-1}$. To this end, we first establish
the following key lemma.

\begin{lemma}\label{lem4.3}
(i) Suppose that for every $v\in V$, $\{\ctP_j v; j=0,1,2,\cdots, J\}$ 
forms a stable decomposition of $\Pad v$.  Then under
assumptions (SA$_0$) and (SA$_1$) the following inequality holds:
\begin{align}\label{eq4.6a}
\sum_{j=0}^J \|\ctP_j v\|_{a_j} \leq \CV  \|\Pad v\|_a \qquad \forall v\in V.
\end{align}

(ii) If the condition of (i) does not hold, then under assumptions 
(SA$_0$)--(SA$_4$) we have
\begin{align}\label{eq4.6b}
\sum_{j=0}^J \|\ctP_j v\|_{a_j} 
&\leq \frac{3\omeW}2 \| \Pad v \|_a 
+ \Bigl[ \bigl(1+N(\Theta) \bigr) \deltaV + \widehat{\delta}_V \Bigr] \| v \|_a,
\end{align}
where $N(\Theta)$ is the same as in Proposition \ref{prop4.1}.

\end{lemma}

\begin{proof}
(i) For any $v\in V$, let $u=\Pad v$, $u_j=\ctP_j v$ for $j=0,1,2,\cdots,J$.  
Since
\begin{align*}
u=\Pad v = \sum_{j=0}^J \cP_j v 
= \sum_{j=0}^J \cR_j^\dag \circ \ctP_j v
=\sum_{j=0}^J \cR_j^\dag u_j,
\end{align*}
$\{u_j\}$ is indeed a decomposition of $u$ which is assumed to 
be stable. By assumption (SA$_1$) we conclude that
\eqref{eq4.1x} holds for $u$, which gives \eqref{eq4.6a}.

(ii) Let $u$ be same as in part (i).  Recall that $\widehat{\cP}=
\sum_{j=1}^J \cP_j$. Using the identities
\begin{align*}
\ctP_0 v &= \frac{1}{2} \Bigl[ \ctP_0 u + \ctP_0 ( v-\cP_0 v) 
           + \ctP_0 (v-\widehat{\cP} v) \Bigr], \\
\ctP_j v &= \frac{1}{J+1} \Bigl[ \ctP_j u
+\sum_{i=0}^J \ctP_j ( v-\cP_i v) \Bigr] \qquad \mbox{for } j = 1, 2, \cdots, J,
\end{align*}
the triangle inequality, $(SA_4)$ and \eqref{eq4.10a} we get 
\begin{align*}
\| \ctP_0 v \|_{a_0} & \leq \frac{1}{2} \Bigl[ \| \ctP_0 u \|_{a_0} 
+ \| \ctP_0 \bigl(v - \cP_0v \bigr) \|_{a_0} 
+ \| \ctP_0 \bigl( v - \widehat{\cP} v) \|_{a_0} \Bigr] \\
&\leq \frac12 \Bigl[ \omeW \| u \|_a 
+ \bigl(\deltaV + \widehat{\delta}_V\bigr) \|v\|_a  \Bigr], \\
\| \ctP_j v \|_{a_j} &\leq \frac{1}{J+1} \Bigl[ \|\ctP_j u\|_{a_j} 
+ \sum_{i=0}^J  \|\ctP_j ( v-\cP_i v)\|_{a_j} \Bigr] \\
&\leq \frac{1}{J+1} \Bigl[ \omeW \|u\|_{a} 
+ \sum_{i=0}^J \theta_{ij} \deltaV \|v\|_{a} \Bigr] \\
&\leq \frac{1}{J+1} \Bigl[ \omeW \|u\|_{a} 
+ N_j(\Theta) \deltaV \| v \|_{a} \Bigr] \qquad\mbox{for } j = 1,2, \cdots, J.
\end{align*}

Then summing the above inequality we obtain
\begin{align*}
\sum_{j = 0}^J \| \ctP_j v \|_{a_j} &\leq \frac{3\omeW}2 \| u \|_{a} 
+ \Bigl[ \bigl(1+N(\Theta) \bigr) \deltaV + \widehat{\delta}_V \Bigr] \| v \|_a.
\end{align*}
Hence, \eqref{eq4.6b} holds. The proof is complete.
\end{proof}

We now are ready to establish a lower bound estimate for the 
additive Schwarz operator $\Pad$.

\begin{proposition}\label{prop4.2}
(i) Under assumptions of (i) of Lemma \ref{lem4.3},  the following estimate holds:
\begin{align}\label{eq4.7}
\|\Pad v\|_a \geq (\CV \CW)^{-1} \|v\|_a \qquad\forall v\in V.
\end{align}

(ii) Under assumptions of (ii) of Lemma \ref{lem4.3},   the following estimate holds:
\begin{align}\label{eq4.8}
\|\Pad v\|_a &\geq K_0^{-1} \|v\|_a \qquad\forall v\in V \end{align}
provided that $C_W\bigl[ \bigl(1+N(\Theta) \bigr) \deltaV 
+ \widehat{\delta}_V \bigr]<1$ where
\begin{align}\label{eq4.8a}
K_0:= \frac{3\omeW \CW }{2-2\CW \bigl[\bigl(1+N(\Theta) \bigr) \deltaV 
+ \widehat{\delta}_V \bigr]}.
\end{align}
Consequently, operator $\Pad$ is invertible.
\end{proposition} 

\begin{proof} 
For any $w\in W$, let $\{w_j\}$ be an energy stable decomposition 
of $w$, that is,  
\[
\qquad\quad w=\sum_{j=0}^J \cS_j^\dag w_j,
\]
and \eqref{eq4.1y} holds.  Then we have
\begin{alignat}{2}\label{eq4.9}
a(v,w) &= \sum_{j=0}^J a(v, \cS_j^\dag w_j) && \\
&= \sum_{j=0}^J a_j(\ctP_j v, w_j) \nonumber &&\qquad (\mbox{by \eqref{eq3.22}}) \\
&\leq \sum_{j=0}^J \|\ctP_j v\|_{a_j} \normWj{w_j} \nonumber 
&&\qquad (\mbox{by \eqref{eq4.5c}}) \\
&\leq \sum_{j=0}^J \|\ctP_j v\|_{a_j}\, \sum_{j=0}^J \normWj{w_j}  \nonumber 
&&\qquad (\mbox{by discrete Schwarz inequality}) \\
&\leq \CW \normW{w}  \sum_{j=0}^J \|\ctP_j v\|_{a_j}. \nonumber 
&&\qquad (\mbox{by \eqref{eq4.1y}}) 
\end{alignat}
The desired estimates \eqref{eq4.7} and \eqref{eq4.8} follow from
substituting \eqref{eq4.6a} and \eqref{eq4.6b} into \eqref{eq4.9}, respectively.  
The proof is complete.  
\end{proof}

\begin{remark}\label{rem4.2}
We note that the argument used in the proof of lower bound 
estimate \eqref{eq4.8} is in the spirit of the so-called
Schatz argument (cf. \cite{BS08}) which is often used to derive
finite element error estimates for nonsymmetric and indefinite problems. 
It is interesting to see that a similar argument also
plays an important role in our Schwarz convergence theory.
\end{remark}

Combining Propositions  \ref{prop4.1} and \ref{prop4.2} we 
obtain our first main theorem of this paper.

\begin{theorem}\label{thm4.1}
(i) If every $v\in V$ has a unique decomposition 
$v=\sum_{j=0}^J \cR_j^\dag v_j$ with $v_j\in V_j$, then under 
assumptions (SA$_0$)--(SA$_3$) the following condition
number estimate holds:  
\begin{align}\label{eq4.16}
\kappa_a(\Pad)
\leq \omeV\omeW \CV \CW \bigl[1+ \omeW\CW N(\Theta)\bigr]. 
\end{align}

(ii) If the above unique decomposition assumption does not hold, then
under assumptions (SA$_0$)--(SA$_4$) the following condition
number estimate holds: 
\begin{align}\label{eq4.17}
\kappa_a(\Pad)
\leq \omeV \omeW \bigl[1+ \omeW\CW N(\Theta)\bigr] K_0. 
\end{align}
Where
\begin{align}\label{eq4.18x}
\kappa_a(\Pad) &:= \|\Pad\|_a \|\Pad^{-1}\|_a, \\
\|\Pad\|_a &:=\sup_{0\neq v\in V} \frac{\|\Pad v\|_a}{\|v\|_a}.
\label{eq4.18y}
\end{align}
\end{theorem}

The above condition number estimates for the operator $\Pad$ 
also translates to its matrix representation. 

\begin{theorem}\label{thm4.2}
(i) Under assumptions of (i) of Theorem \ref{thm4.1} the following condition
number estimate holds:
\begin{align}\label{eq4.20}
\kappa_A(P_{\mbox{\tiny ad}}) 
\leq \omeV\omeW \CV \CW \bigl[1+ \omeW\CW N(\Theta)\bigr]. 
\end{align}

(ii) Under assumptions of (ii) of Theorem \ref{thm4.1} the following condition
number estimate holds:
\begin{align}\label{eq4.21}
\kappa_A(P_{\mbox{\tiny ad}}) 
\leq \omeV \omeW \bigl[1+ \omeW\CW N(\Theta)\bigr] K_0,
\end{align}
where 
\begin{align}\label{eq4.22x}
\kappa_A(P_{\mbox{\tiny ad}}) &:= \| P_{\mbox{\tiny ad}}\|_A 
\| P_{\mbox{\tiny ad}}^{-1} \|_A, \\
\| P_{\mbox{\tiny ad}}\|_A &:=\sup_{0\neq \bv\in \bR^d} 
\frac{ \|P_{\mbox{\tiny ad}}\bv\|_A }{ \|\bv\|_A }, \\
\|\bv\|_A &:=\sqrt{A\bv\cdot A\bv} =\sqrt{A^T A\bv\cdot \bv}. \label{eq4.22z}
\end{align}
\end{theorem}

\subsection{Condition number estimate for $\cP_{\mbox{\tiny hy}}$}
\label{sec-4.3}
As in the case of SPD problems \cite[section 2.5.2]{TW05},  
we replace the structure assumption (SA$_1$) by the following one:

\medskip
\begin{itemize}
\item[($\widetilde{\mbox{SA}}_1$)] {\em Energy stable decomposition 
assumption.} There exist positive constants $\tCV$ 
and $\tCW$ such that every pair $(\varphi,\psi)\in \mbox{range}(\cI-\alpha\cP_0)
\times \mbox{range}(\cI-\alpha\cQ_0)$ admits a decomposition 
\[
\varphi=\sum_{j=1}^J \cR_j^\dag \varphi_j,
\qquad\quad \psi=\sum_{j=1}^J \cS_j^\dag \psi_j,
\]
with $\varphi_j\in V_j$ and $\psi_j\in W_j$ such that
\begin{align}\label{eq4.23x}
\sum_{j=1}^J \|\varphi_j\|_{a_j} &\leq \tCV \|\varphi\|_a, \\
\sum_{j=1}^J \normWj{\psi_j} &\leq \tCW \normW{\psi}.  \label{eq4.23y}
\end{align}
\end{itemize}
We remark that the new {\em energy stable decomposition assumption} 
($\widetilde{\mbox{SA}}_1$)
implies that any pair $(v,w)\in V\times W$ has a stable 
decomposition (in the sense of (SA$_1$)) of the following form:
\[
v=\alpha\cP_0 v + \sum_{j=1}^J \cR_j^\dag \varphi_j,\qquad
w=\alpha\cQ_0 w + \sum_{j=1}^J \cS_j^\dag \psi_j,
\]
where $\{(\varphi_j,\psi_j)\}_{j=1}^J$ is a stable decomposition 
(in the sense of ($\widetilde{\mbox{SA}}_1$)) for 
$\bigl( (\cI-\alpha\cP_0)v, (\cI-\alpha\cQ_0)w \bigr)$.

Next lemma shows that $\cP_j$ (resp. $\Pad$) and $\cQ_j$ (resp. $\Qad$)
are mutually conjugate with respect to the bilinear form $a(\cdot,\cdot)$.

\begin{lemma}\label{lem4.8}
The following identities hold:
\begin{align}\label{e4.24x}
a(\cP_j v,w) &= a(v,\cQ_j w) \qquad\forall (v,w)\in V\times W, 
\, j=0,1,2,\cdots,J. \\
a(\Pad v,w) &= a(v,\Qad w) \qquad\forall (v,w)\in V\times W.
\end{align}
\end{lemma}
Since the proof is trivial, we omit it to save space.

The following proposition is the analogue to Proposition \ref{prop4.1}
for the hybrid operator $\cP_{\mbox{\tiny hy}}$. 

\begin{proposition}\label{prop4.3}
Under assumptions (SA$_0$), ($\widetilde{\mbox{SA}}_1$), (SA$_2$) and 
(SA$_3$) the following estimate holds:
\begin{align}\label{eq4.25}
\|\cP_{\mbox{\tiny hy}} v\|_a 
\leq  \omeV \omeW \bigl[ \alpha+ \omeW \tCW N(\widehat{\Theta}) 
\bigl(1+ \alpha\omeV \omeW C_{a_j} \beta_{a_j}^{-1} \bigr) 
\bigr] \|v\|_a  
\end{align} 
for all $v\in V$.  Where $\widehat{\Theta}=\Theta(1:J,1:J)$. 

\end{proposition}

\begin{proof}
Let $\widehat{\cP}:=\sum_{j=1}^J \cP_j$ and $\widehat{\cQ}:=\sum_{j=1}^J \cQ_j$.
For any $v\in V$ and $w\in W$, let $\varphi:=(\cI-\alpha\cP_0) v$ and 
$\psi:=(\cI-\alpha\cQ_0)w$. Obviously, 
$\varphi\in \mbox{range}(\cI-\alpha\cP_0)$
and $\psi \in \mbox{range}(\cI-\cQ_0)$.
By assumption ($\widetilde{\mbox{SA}}_1$), $(\varphi, \psi)$ admits 
an energy stable decomposition $\{(\varphi_j,\psi_j)\}_{j=1}^J$. 
Thus,
\begin{align}\label{eq4.26}
a((\cI-\alpha\cP_0)\widehat{\cP} v,w) 
&= a(\widehat{\cP} v, (\cI-\alpha\cQ_0)w) \\
&= a(\widehat{\cP} v, \psi) \nonumber\\
&= \sum_{i=1}^J\sum_{j=1}^J  a(\cR_i^\dag \ctP_i v, \cS_j^\dag \psi_j) 
\nonumber \\
&\leq \sum_{i=1}^J\sum_{j=1}^J \theta_{ij}\|\cP_i v\|_a\normW{\cS_j^\dag\psi_j} 
\nonumber \\
&\leq \omeV \omeW  \|v\|_a \sum_{j=1}^J N_j(\widehat{\Theta})  
\normW{\cS_j^\dag \psi_j}  \nonumber \\
&\leq \omeV \omeW^2   N(\widehat{\Theta}) \|v\|_a  
\sum_{j=1}^J \normWj{\psi_j}   \nonumber \\
&\leq \omeV \omeW^2 \tCW N(\widehat{\Theta}) \|v\|_a \normW{\psi} \nonumber \\
&\leq \omeV \omeW^2 \tCW N(\widehat{\Theta}) 
\bigl(1+\alpha \omeV \omeW C_{a_j} \beta^{-1}_{a_j} \bigr) 
\|v\|_a \normW{w},  \nonumber
\end{align}
where we have used \eqref{eq4.10f} to obtain the last inequality.
The above inequality in turn implies that
\begin{align*} 
\|(\cI-\alpha\cP_0)\widehat{\cP} v\|_a 
\leq \omeV \omeW^2 \tCW N(\widehat{\Theta}) 
\bigl(1+ \alpha \omeV \omeW C_{a_j} \beta^{-1}_{a_j} \bigr) \|v\|_a,
\end{align*}
and 
\begin{align*}
\|\cP_{\mbox{\tiny hy}} v\|_a 
&\leq \alpha\|\cP_0 v\|_a + \|(\cI-\alpha\cP_0)\widehat{\cP} v\|_a \\
&\leq \alpha\omeV \omeW + \omeV \omeW^2 \tCW N(\widehat{\Theta}) 
\bigl(1+ \alpha \omeV \omeW C_{a_j} \beta^{-1}_{a_j}\bigr) \|v\|_a.
\end{align*}
Hence, \eqref{eq4.25} holds and the proof is complete.
\end{proof}

Next, we derive a lower bound estimate for $\|\cP_{\mbox{\tiny hy}}\|_a$.
The following proposition is an analogue of Proposition \ref{prop4.2}. 

\begin{proposition}\label{prop4.4}
Under assumptions (SA$_0$), ($\widetilde{\mbox{SA}}_1$), (SA$_2$)--(SA$_4$), 
along with the assumption $\mbox{range} \bigl( I - \alpha\cQ_0 \bigr) = W$ 
the following estimate holds:
\begin{align}\label{eq4.27a}
\|\cP_{\mbox{\tiny hy}} v\|_a 
&\geq K_1^{-1} \|v\|_a \qquad\forall v\in V, 
\end{align}
provided that $\deltaV\tCW N\bigl(\widehat{\Theta}\bigr)
+\alpha \omeV \widehat{\delta}_V<J$.  Where
\begin{align}\label{eq4.27b}
K_1:= \frac{1}{J - \deltaV\tCW N\bigl(\widehat{\Theta}\bigr)
-\alpha \omeV \widehat{\delta}_V}.
\end{align}
Consequently, operator $\cP_{\mbox{\tiny hy}}$ is invertible.
\end{proposition}

\begin{proof}
For any $v\in V$ and $w\in W$. Let $\psi:=(\cI-\alpha \cQ_0)w\in 
\mbox{range}(\cI-\alpha \cQ_0)$ and $u:=\cP_{\mbox{\tiny hy}} v$. 
Assumption ($\widetilde{\mbox{SA}}_1$)
ensures that $\psi$ has an energy stable decompositions
$\{\psi_j\}_{j=1}^J$ with $\psi_j\in W_j$, that is, 
\begin{alignat}{2}\label{eq4.28}
&\psi= \sum_{j=1}^J \cS_j^\dag \psi_j & \qquad \mbox{and} 
&\qquad \sum_{j=1}^J \normWj{\psi_j} \leq \tCW  \normW{\psi}.
\end{alignat}
Using the following identity
\begin{align*}
v = \frac{1}{J} \Bigl[ u + \sum^{J}_{i = 1} \bigl(v - \cP_i v \bigr) 
+ \alpha \cP_0 \bigl( \widehat{\cP}v - v \bigr) \Bigr],
\end{align*}
(SA$_4$), \eqref{eq4.3x} and \eqref{eq4.28} we get
\begin{align*}
a(v,\psi) &= \frac{1}{J} \Bigl[ a(u, \psi) 
+ \sum^{J}_{i = 1}a \bigl(v - \cP_i v, \psi \bigr) 
+ \alpha a\bigl(\cP_0 \bigl(\widehat{\cP}v - v\bigr), \psi\bigr) \Bigr] \\
& \leq \frac{1}{J} \Bigl[ \| u \|_a  \normW{\psi} 
+ \sum^J_{j = 1} \sum^{J}_{i = 1}a \bigl(v - \cP_i v, \cS_j^\dag \psi_j \bigr) 
+ \alpha\bigl\| \cP_0 \bigl(\widehat{\cP}v -v\bigr)\bigr\|_a \normW{\psi }\Bigr]\\
& \leq \frac{1}{J} \Bigl[ \| u \|_a \normW{\psi} 
+ \sum^J_{j = 1} \sum^{J}_{i = 1}a_j \bigl( \ctP_j \bigl(v - \cP_i v \bigr), \psi_j \bigr) 
+\alpha \omeV \bigl \| \ctP_0 \bigl(\widehat{\cP}v - v\bigr) \bigr \|_{a_0}\normW{\psi} \Bigr] \\
& \leq  \frac{1}{J} \Bigl[\| u \|_a \normW{\psi} 
+ \sum^J_{j = 1} \sum^{J}_{i = 1}\bigl \| \ctP_j \bigl(v-\cP_i v\bigr) \bigr\|_{a_j}\normWj{\psi_j} 
+ \alpha\omeV \widehat{\delta}_V \| v \|_a \normW{\psi} \Bigr] \\
& \leq \frac{1}{J} \Bigl[ \| u \|_a \normW{\psi} 
+ \sum^J_{j = 1} \sum^{J}_{i = 1} \theta_{ij} \deltaV \|v \|_{a} \normWj{\psi_j}
+ \alpha\omeV \widehat{\delta}_V  \| v \|_a \normW{\psi} \Bigr] \\
& \leq \frac{1}{J} \Bigl[ \| u \|_a \normW{\psi} 
+ \sum^J_{j = 1} N_j \bigl(\widehat{\Theta} \bigr) \deltaV \|v \|_{a} \normWj{\psi_j} 
+ \alpha\omeV \widehat{\delta}_V  \| v \|_a\normW{\psi} \Bigr] \\
& \leq \frac{1}{J} \Bigl[ \| u \|_a \normW{\psi} 
+ N\bigl(\widehat{\Theta}\bigr) \deltaV \|v \|_{a} \sum^J_{j = 1}  
\normWj{\psi_j}
+ \alpha\omeV \widehat{\delta}_V  \| v \|_a \normW{\psi} \Bigr] \\
& \leq \frac{1}{J} \Bigl[ \| u \|_a \normW{\psi} 
+ \tCW N \bigl(\widehat{\Theta} \bigr) \deltaV \|v \|_{a}  \normW{\psi} 
+ \alpha\omeV \widehat{\delta}_V  \| v \|_a \normW{\psi} \Bigr].
\end{align*}
The desired estimate follows from the assumption 
$\mbox{range}(I -\alpha \cQ_0) = W$.
\end{proof}

\begin{remark}
We note that the assumption $\mbox{\rm range }(\cI-\alpha\cQ_0)=W$ is 
equivalent to asking $\cI-\alpha\cQ_0$ to be invertible, which holds
for sufficiently small or large relaxation parameter $\alpha$.
\end{remark}

Combining Propositions  \ref{prop4.3} and \ref{prop4.4} we
obtain our third main theorem of this paper.

\begin{theorem}\label{thm4.3}
Under assumptions (SA$_0$)--(SA$_4$) and 
$\mbox{range} \bigl( I - \alpha\cQ_0 \bigr) = W$ 
the following condition number estimate holds:
\begin{align}\label{eq4.50}
\kappa_a(\cP_{\mbox{\tiny hy}})
\leq \omeV \omeW \bigl[ \alpha+ \omeW \tCW N(\widehat{\Theta}) 
\bigl(1+ \alpha\omeV \omeW C_{a_j} \beta_{a_j}^{-1} \bigr) 
\bigr]\, K_1.
\end{align}
\end{theorem}


\section{Application to DG discretizations for convection-diffusion problems}
\label{sec-5}
In this section we shall use our abstract framework and the 
abstract convergence theory developed in Sections \ref{sec-3}
and \ref{sec-4} to construct three types of Schwarz methods
for discontinuous Galerkin approximations
of the following general diffusion-convection problem:
\begin{alignat}{2}\label{eq5.1}
\cL u:=-\Div(D(x)\nab u) + \bb(x)\cdot \nab u +c(x) u &=f 
&&\qquad\mbox{in }\Ome,\\
u &=0 &&\qquad\mbox{on } \p\Ome, \label{eq5.2}
\end{alignat}
where $\Ome \subset \bR^d \, (d=1,2,3)$ is a bounded domain with 
Lipschitz continuous boundary $\p\Ome$. $D(x)\in \bR^{d\times d}$ 
satisfies $\lambda |\xi|^2 \leq D(x) \xi\cdot\xi\leq \Lambda |\xi|^2$
$\,\forall\xi\in \bR^d$ for some positive constants $\lambda$ and $\Lambda$. 
So \eqref{eq5.1} is uniformly elliptic in $\Ome$ \cite[Chapter 8]{GT00}.
Assume that $\bb\in H(\mbox{div, } \Ome)$ or $\bb\in [C^0(\overline{\Ome})]^d$, 
$c\in L^\infty(\Ome)$
and $f\in L^2(\Ome)$. Let $V=W=H^1_0(\Ome)$, then the variational 
formulation of \eqref{eq5.1}--\eqref{eq5.2} is defined as \cite{BA72,GT00}
\begin{align}\label{eq5.3}
\cA(u,w)=\cF(w) \qquad\forall w\in W,
\end{align}
where
\begin{align}\label{eq5.4}
\cA(u,w) &:= \int_\Ome \Bigl( D(x)\nab u\cdot \nab w + \bb(x)\cdot \nab u w
+c(x) u w \Bigr)\, dx, \\
\cF(w) &:=\int_\Ome f w\, dx. \label{eq5.5}
\end{align}

Clearly, when $\bb(x)\not\equiv 0$, the bilinear form $\cA(\cdot, \cdot)$ 
is nonsymmetric. The problem can be further classified as follows:

\medskip
(i) {\em Positive definite case:}  If $\bb$ and $c$ satisfies
\begin{align}\label{eq5.6}
c(x)-\frac12 \Div \bb(x) \geq 0 \qquad \mbox{in } \Ome.
\end{align}

\medskip
(ii) {\em Indefinite case:} If $\bb$ and $c$ satisfies
\begin{align}\label{eq5.7}
c(x)-\frac12 \Div \bb(x) < 0 \qquad \mbox{in } \Ome.
\end{align}

It is easy to check that all the conditions of 
the classical Lax-Milgram Theorem hold in the {\em positive definite case}. 
It also can be shown \cite{BA72} that
in the {\em indefinite case} all the conditions of Theorem \ref{babuska} are 
satisfied provided that problem \eqref{eq5.1}--\eqref{eq5.2} and its 
adjoint problem are uniquely solvable for arbitrary source terms. 
It is also well known \cite{BA72,GT00} that in {\em indefinite case}
the bilinear form $\cA(\cdot, \cdot)$ satisfies a G{\aa}rding-type inequality
instead of the strong coercivity.

\subsection{Discontinuous Galerkin approximations}\label{sec-5.1}
Consider a special case of \eqref{eq5.1} where $D(x) = \epsilon > 0$, 
$\mathbf{b} \in [W^{1,\infty}(\Ome)]^d$, and $c(x) = \Div(\mathbf{b}(x)) 
+ \gamma(x)$ where $\gamma \in L^\infty(\Ome)$.  To discretize 
this problem, we shall use an interior penalty discontinuous
Galerkin (IPDG) scheme developed in \cite{AD09}. For this scheme 
we require a shape-regular triangulation $\mathcal{T}_h$ of the 
domain $\Ome$.  The scheme can then be written in the form \eqref{eq2.10} where
\begin{align}
&V = W := \left\{ v \in L^{2}(\Ome) \mbox{ such that } v|_{K} 
\in \mathbb{P}_r(K) \ \forall K \in \mathcal{T}_h  \right\}, \label{eq5.8} \\
&a(u, w) := \sum_{K \in \mathcal{T}_h} \int_{K} (\gamma u w 
+( \epsilon \nabla u - \mathbf{b}u) \cdot \nabla w)\, dx 
+ \sum_{e \in \Gamma} c_e \frac{\epsilon}{| e |} \int_e [u] \cdot [w]\, ds 
\label{eq5.9} \\
&\quad+ \sum_{e \in \mathcal{E}^{\circ}_h} \int_e \{\mathbf{b}u\}_{upw} 
\cdot [w]\, ds 
- \sum_{e \in \Gamma} \int_{e} \{\epsilon \nabla_hu\} \cdot [w] \, ds
+ \sum_{e \in \Gamma^+} \int_{e} \mathbf{b \cdot n}uw\, ds, \notag \\
&f(w) := \sum_{K \in \mathcal{T}_h} \int_K f w\, dx.  \label{eq5.10}
\end{align}
Where $r\geq 1, \Gamma = \partial \Ome$, $\mathbf{n}$ is the unit outward 
normal vector to $\Gamma$, and $\Gamma^+$ indicates the outflow portion of 
$\Gamma$ defined as
\begin{align*}
\Gamma^{+} = \left\{ x \in \Gamma \mbox{ such that } \mathbf{b}(x) 
\cdot \mathbf{n}(x) \geq 0 \right\}.
\end{align*}
$[\cdot]$ and $\{\cdot\}$ are the standard jump and average operators, respectively, 
and $\{\cdot\}_{upw}$ is the upwind flux.  
To define this flux, we consider a vector valued function $\boldsymbol{\tau}$ 
defined on two neighboring elements $K_1$ and $K_2$ of $\mathcal{T}_h$ 
with common edge $e$.  Suppose that $\boldsymbol{\tau}^i 
= \boldsymbol{\tau}|_{K_i}$ for $i = 1, 2$.  Then 
$\{ \boldsymbol{\tau} \}_{upw}$ is defined on the edge $e$ as follows:
\begin{align*}
\{ \boldsymbol{\tau} \}_{upw} = \frac{1}{2}(\mbox{sign}(\mathbf{b} \cdot 
\mathbf{n}^{1}) + 1) \boldsymbol{\tau}^1 + \frac{1}{2}(\mbox{sign}(\mathbf{b} 
\cdot \mathbf{n}^{2}) + 1) \boldsymbol{\tau}^2,
\end{align*}
where $\mathbf{n}^i$ is the unit outward normal vector of $K_i$ on $e$ 
for $i = 1, 2$.  The choice of this scheme was made because it was 
shown \cite{AD09} that in the positive definite case (i.e. when 
\eqref{eq5.6} holds) this scheme satisfies (MA1) and (MA2) 
(cf. Section \ref{sec-3.2}).  

Once a discretization scheme is chosen we can begin to develop our 
space decomposition and local solvers.  In this example, we will obtain 
the space decomposition by using a nonoverlapping domain decomposition.  
Let $\mathcal{T}_H$ be a coarse mesh of $\Ome$ and $\mathcal{T}_s$ a 
nonoverlapping partition $\left\{ \Ome_j \right\}^J_{j=1}$ of $\Ome$ 
such that $\mathcal{T}_s \subseteq \mathcal{T}_H \subseteq \mathcal{T}_h$.  
Then we define
\begin{align}
&V_0 = W_0 :=  \left\{ v \in L^{2}(\Ome) \mbox{ such that }v|_{K} 
\in \mathbb{P}_r \ \forall K \in \mathcal{T}_H \right\}, \label{eq5.11} \\
&V_j = W_j :=  \left\{ v \in L^2(\Ome_j) \mbox{ such that } v|_{K} 
\in \mathbb{P}_r \ \forall K \in \mathcal{T}_h \mbox{ with } K 
\subseteq \Ome_j \right\}  \label{eq5.12}
\end{align}
for $j = 1, 2, \dots, J$ and $r\geq 1$. For the prolongation operator 
$\mathcal{R}^{\dagger}_0 = \mathcal{S}^{\dagger}_0$ we use the polynomial 
interpolation on each element $K \in \mathcal{T}_h$.
\begin{align}
\mathcal{R}^{\dagger}_0 u_0|_K = \mbox{ the interpolant of } 
u_0 \mbox{ in }\mathbb{P}_r(K)
\end{align}
for each $u_0 \in V_0$ and $K \in \mathcal{T}_h$.
For the prolongation operators $\mathcal{R}^{\dagger}_j 
= \mathcal{S}^{\dagger}_j$ we use the following natural injection into $V$:
\begin{align}
	\mathcal{R}^{\dagger}_j u_j = 
		\left\{\begin{array}{l l}
			u_j & \qquad \mbox{ in } \Ome_j \\
			0 & \qquad \mbox{ in } \Ome \setminus \Ome_j.
		\end{array}\right.
\end{align}
For the local bilinear forms $a_j(\cdot,\cdot)$ we use the exact local 
solvers defined by
\begin{align}
a_j(u_j, w_j) := a( \mathcal{R}^{\dagger}_j u_j, 
\mathcal{R}^{\dagger}_j w_j) \qquad \forall \ u_j, w_j \in V_j \label{eq5.13}
\end{align}
for $j = 0, 1, \dots, J$.
Note that in this example we only have one set of subspaces 
$\{V_j\}^{J}_{j=0}$ and one set of prolongation operators 
$\{\mathcal{R}^{\dagger}_j\}^{J}_{j=0}$ so we shall only have one set of 
projection-like operators $\{\mathcal{P}_j\}^{J}_{j=0}$ defined in 
\eqref{eq3.22} and \eqref{eq3.24}.  Using these projection-like 
operators we can then build the Schwarz operators $\mathcal{P}_{ad}$, 
$\mathcal{P}_{mu}$, and $\mathcal{P}_{hy}$ defined in \eqref{eq3.25}, 
\eqref{eq3.36}, and \eqref{hyb3} respectively.

\subsection{Numerical Experiments}\label{sec-5.2}
In this section we present several $1$-D numerical experiments to gauge 
the theoretical results proved in the previous section.  
For these experiments we concentrated on equation \eqref{eq5.1} in the 
domain $\Ome = (0, 1)$ with the following choices of constant coefficient:

\medskip
{\bf Test 1.} $D(x) = 1$, $b(x) = 1,000$, and $c(x) = 1$.

\smallskip
{\bf Test 2.} $D(x) = 1$, $b(x) = 2,000$, and $c(x) = 1$.

\smallskip
{\bf Test 3.} $D(x) = 1$, $b(x) = 10,000$, and $c(x) = 1$.

\smallskip
{\bf Test 4.} $D(x) = 1$, $b(x) = 100,000$, and $c(x) = 1$.

\medskip
Note that these choices of coefficients put us in the convection dominated 
regime and fit the criteria of the positive definite case characterized 
by \eqref{eq5.6}.  For this reason we are able to use the discretization 
scheme and domain decomposition techniques described in Section \ref{sec-5.1}.
In these experiments we use a uniform fine mesh size $h = 1/256$ and 
a uniform coarse mesh size $H = 1/64$.  The equations are solved 
numerically using standard GMRES, GMRES after using $\mathcal{P}_{ad}$ 
preconditioning, the multiplicative Schwarz iterative method \eqref{eq3.41}, 
and GMRES after using $\mathcal{P}_{hy}$ preconditioning.  To verify the 
dependence of $\kappa_a(\mathcal{P}_{ad})$ and $\kappa_a(\mathcal{P}_{hy})$,  
we use a varying number of subdomains $J = 4, 8, 16, 32, 64$.

Our first goal in these experiments is to compare the performance of the 
Schwarz methods to that of standard GMRES in order to verify the 
usefulness of such methods.  We would like to verify numerically 
that the estimates given in previous sections are tight.  In particular, 
we would like to find an example that shows that 
$\kappa_A(P_{ad})$ does in fact depend linearly on the 
number of subdomains $J$ as predicted in Theorem \ref{thm4.2}.  
For multiplicative Schwarz iteration we would like to estimate 
$\| E_{mu} \|_A$, noting that if this norm is less than $1$ it 
guarantees convergence of the method.  If not, we shall
need to rely on the spectral radius $\rho(E_{mu})$ to guarantee 
this convergence.

Tables \ref{Tab1}--\ref{Tab4} collect the test results on 
the additive, multiplcative, and hybrid Schwarz methods proposed
in Section \ref{sec-3}. Where $J =\mbox{NA}$ represents 
the original system with no preconditioning. From these numerical 
results we can make the following observations:

\begin{enumerate}
\item[(a)] Any of these methods offers an improvement in terms of the CPU 
time needed to solve the system when compared to solving the system using 
standard GMRES.
\item[(b)] GMRES after using $\mathcal{P}_{ad}$ or $\mathcal{P}_{hy}$ for 
preconditioning performs better when the number of subdomains $J$ is 
not too large.
\item[(c)] In all of these tests $\kappa_A(P_{ad})$ and $\kappa_A(P_{hy})$ 
depend on the number of subdomains $J$.  Particularly in \textbf{Test 2}, 
we see an example that exhibits approximate linear dependence. 
See figure \ref{fig6.1}.
\item[(d)] For $\| E_{mu} \|_A$ we do not observe such a strong dependence 
on the number of subdomains $J$.
\item[(e)] In these tests $\| E_{mu} \|_A$ is greater than $1$; thus, we 
cannot rely upon this as an indicator for convergence of the multiplicative 
Schwarz iterative method.
\item[(f)] $\kappa_A$ is not a unique metric in judging the convergence of 
GMRES after preconditioning with $P_{ad}$ and $P_{hy}$.  For instance,
in \textbf{Test 4} $\kappa_A(P_{ad})$ decreases while the number of 
iterations necessary for GMRES increases as $J$ increases.  This 
is opposite of the behavior that is observed in the previous tests.
\end{enumerate}

\medskip
Our numerical experiments verify that $\kappa_A$ is not a unique metric
for the convergence of GMRES.  Therefore, we must rely on other metrics
to predict the convergence behavior of GMRES.  The following theorem is
used in \cite{TW05} to test convergence of GMRES after Schwarz
preconditioning in the indefinite case:
\begin{theorem}[\cite{EES83}]\label{thm6.1}
Consider the linear system $A \mathbf{x} = \mathbf{b}$ where
$A \in \mathbb{R}^{d \times d}$ and $\mathbf{x}, \mathbf{b} \in \mathbb{R}^d$.
If the symmetric part of $A$ is positive definite, then after $k$ step of GMRES, 
the norm of the residual $\mathbf{r_k}:=\mathbf{b}-A \mathbf{x}^{(k)}$
is bounded by
\begin{align*}
\|\mathbf{r_k} \|_2 \leq \left(1 - \frac{c_p^2}{C_p^2}\right)^{k/2} 
\| \mathbf{r_0} \|_2, 
\end{align*}
where $c_p>0$ is the minimal eigenvalue of the symmetric part of $A$
and $C_p$ is the operator norm of $A$ given by
\begin{align*}
c_p := \min_{\mathbf{u} \in \mathbb{R}^d} \frac{\langle \mathbf{u}, 
A \mathbf{u} \rangle}{\langle \mathbf{u} , \mathbf{u}\rangle}, 
\qquad C_p := \max_{\mathbf{u} \in \mathbb{R}^d} 
\frac{ \|A\mathbf{u} \|_2}{\|\mathbf{u} \|_2}.
\end{align*}
\end{theorem}

Unfortunately, this theorem cannot be applied in our case because we are
not guaranteed that $P_{ad}$ and $P_{hy}$ have positive definite symmetric
part (i.e., $c_p>0$).  In the tests previously done, we find that
these operators can be indefinite (i.e., $c_p<0$). Another result that
could be of help in this area is the following theorem (cf. \cite{TB97}).

\begin{theorem}\label{thm6.2}
Consider the linear system $A \mathbf{x} = \mathbf{b}$ where
$A \in \mathbb{R}^{d \times d}$ and $\mathbf{x}, \mathbf{b} \in \mathbb{R}^d$.  Further suppose that $A$ is diagonalizable.  Then after $k$ steps of
GMRES, the residual $\mathbf{r}_k:=\mathbf{b}-A \mathbf{x}^{(k)}$ satisfies
\begin{align*}
\frac{\| \mathbf{r}_k \|_2}{\| \mathbf{b} \|_2} 
\leq \kappa_2(V) \inf_{\substack{p \in \mathbb{P}_k \\ p(0) = 1 }} 
\sup_{\lambda \in \sigma(A)} | p (\lambda) |,
\end{align*}
where $V$ is a nonsingular matrix of eigenvectors of $A$ and $\sigma (A)$
denotes the spectrum of $A$.
\end{theorem}

The above theorem says that the spread of the spectrum is a metric to
judge the performance of GMRES with GMRES performing better when the
spectrum of $A$ is clustered.  With this theorem in mind, let us examine
the spectrum of the matrix $A$ and $P_{ad}$ for $J = 4, 8, 16, 32, 64$
obtained in \textbf{Test 2} and \textbf{Test 4}.

Note that in Figure \ref{fig6.2}(a) and Figure \ref{fig6.3}(a) the spectrum
has a large spread which is consistent with the fact that GMRES performed
poorly on the original system without preconditioning.  We also see that
after preconditioning, the spectrum of $P_{ad}$ is clustered which
corresponds to improved performance of GMRES after preconditioning
with $P_{ad}$.  Lastly, we note that as the number of subdomains $J$
increases, the spread of the spectrum of $P_{ad}$ increases.  This
corresponds to a decreased performance in GMRES after preconditioning
with $P_{ad}$ as $J$ increases.

This result leads us to believe that to accurately judge the
behavior of GMRES after preconditioning one needs to analyze the spectrum of
the preconditioned system.  Similarly, we find that to accurately predict
the performance of the multiplicative Schwarz iterative method one needs to
analyze the spectral radius of $E_{mu}$.

Another property of $\kappa_A(P_{ad})$ that is of interest is
its dependence on the fine mesh parameter $h$ and the coarse mesh parameter $H$.
It was shown in \cite{FK01} that $\kappa_A(P_{ad}) = O(\frac{H}{h})$
using two-level non-overlapping domain decomposition for an IPDG
approximation of the equation \eqref{eq5.1} with $\mathbf{b} = 0$
and $c = 0$. This work uses the existing framework that is laid down
for the symmetric and positive definite case.  We would like to test
our new framework to see if this dependence on $H$ and $h$ is still true.
For this reason in {\bf Test 1 - 4} we have calculated
$\kappa_A(P_{ad})$ with $1/h = 16, 32, 64, 128$ and $1/H = 4, 8, 16$
when $J = 4$. We also calculated $\kappa_A(P_{ad})$ with
$1/h = 32, 64, 128, 256$ and $1/H = 8, 16, 32$ when $J = 8$.
From Table \ref{table 5.5} and Table \ref{table5.6} it seems
that we cannot expect $\kappa_A(P_{ad}) = O(\frac{H}{h})$ but
instead we might expect the dependence to be one of the type
\begin{align*} 
\kappa_A(P_{ad}) = O\left(\frac{H^{\sigma_1}}{h^{\sigma_2}}\right)
\end{align*}
where $\sigma_1$ and $\sigma_2$ are positive real numbers
with $\sigma_1 < \sigma_2$ or some other more complicated dependence on
$H$ and $h$.

\vfill\eject
\begin{table}[t] 
\centering
	\subfloat[GMRES after preconditioning with $P_{ad}$ and $P_{hy}$]{
	\begin{tabular}{| c | c | c | c | c | c | c |}
	\hline
	J   & \multicolumn{2}{ | c | }{ Iteration \# } & \multicolumn{2}{ | c | }{CPU Time }& \multicolumn{2}{ | c | }{$\kappa_A$} \\
   	 & \multicolumn{2}{ | c | }{of GMRES} & \multicolumn{2}{ | c | }{ } & \multicolumn{2}{| c |}{ } \\ \hline
	NA & \multicolumn{2}{| c |}{552} & \multicolumn{2}{| c |}{14.3760} & \multicolumn{2}{| c |}{$3.3893 \times 10^4$} \\ \hline
	& $P_{ad}$ & $P_{hy}$ & $P_{ad}$ & $P_{hy}$ & $P_{ad}$ & $P_{hy}$ \\ \hline
	4   & 7  &  2 & 1.3638 & 1.1922  & 460.5713 & 397.3567 \\ \hline
	8   & 7   &  3 & 1.3343 & 1.2367  & 436.7967 & 398.2544 \\ \hline
	16 & 11  &  5 & 1.6873 & 1.4040  & 438.2207 & 412.1700 \\ \hline
	32 & 17  &  8 & 2.6431 & 1.9066  & 521.3530 &  478.9537 \\ \hline
	64 & 30  &  15 & 6.2315 & 3.7889  & 774.7091 & 619.3973 \\ \hline
	\end{tabular}} \\
	\subfloat[Multiplicative Schwarz Iteration]{
	\begin{tabular} {| c | c | c | c | c |}
		\hline
		J & Iterations \# of & CPU Time& $ \| E_{mu} \|_A$ & $\rho(E_{mu})$ \\  
		  & Mult. Schwartz & & & \\ \hline
		4 & 2 & 1.1060 & 19.8830 & $4.4793 \times 10^{-6}$ \\ \hline
		8 & 2 & 1.1016 & 19.8889 & 0.0029 \\ \hline
		16 & 3 & 1.1352 & 19.8469 & 0.0725 \\ \hline
		32 & 5 & 1.2768 & 19.7658 & 0.3179 \\ \hline
		64 & 8 & 1.7129 & 19.7176 & 0.5926 \\ \hline
	\end{tabular}
	}
	\caption{Performance of three Schwarz methods on \textbf{Test 1} } 
\label{Tab1}
\end{table} 

\begin{table}[b] 
\centering
\subfloat[GMRES after preconditioning with $P_{ad}$ and $P_{hy}$]{
	\begin{tabular}{| c | c | c | c | c | c | c |}
	\hline
	J   & \multicolumn{2}{ | c | }{ Iteration \# } & \multicolumn{2}{ | c | }{ CPU Time }& \multicolumn{2}{ | c | }{$\kappa_A$} \\
   	 & \multicolumn{2}{ | c | }{of GMRES} & \multicolumn{2}{ | c | }{ } & \multicolumn{2}{| c |}{ } \\ \hline
	NA & \multicolumn{2}{| c |}{550} & \multicolumn{2}{| c |}{14.4971} & \multicolumn{2}{| c |}{$1.7388 \times 10^4$} \\ \hline
	& $P_{ad}$ & $P_{hy}$ & $P_{ad}$ & $P_{hy}$ & $P_{ad}$ & $P_{hy}$ \\ \hline
	4   & 8 & 3 & 1.3249 & 1.2069 & 741.9511 & 699.5729 \\ \hline
	8   & 10 & 5 & 1.4463 & 1.2835 & 749.0976 & 713.3674 \\ \hline
	16 & 17 & 8 & 1.9924 & 1.5557 & 847.4815 & 800.9121 \\ \hline
	32 & 27 & 14 & 5.5602 & 2.4255 & $1.1221 \times 10^3$ & $1.0029 \times 10^3$ \\ \hline
	64 & 44 & 24 & 8.7063 & 5.3089 & $1.6247 \times 10^3$ & $1.2918 \times 10^3$ \\ \hline
	\end{tabular}}\\
	\subfloat[Multiplicative Schwarz Iteration]{
	\begin{tabular} {| c | c | c | c | c |}
		\hline
		J & Iterations \# of & CPU Time& $ \| E_{mu} \|_A$ & $\rho(E_{mu})$ \\  
		  & Mult. Schwartz & & & \\ \hline
		4 & 2 & 1.1010 & 26.4005 & 0.0011 \\ \hline
		8 & 3 & 1.1131 & 26.3187 & 0.0451 \\ \hline
		16 & 4 & 1.1679 & 26.1222 & 0.2529 \\ \hline
		32 & 6 & 1.3214 & 25.9832 & 0.5277 \\ \hline
		64 & 10 & 1.8713 & 25.9270 & 0.7167 \\ \hline
	\end{tabular}}
	\caption{Performance of three Schwarz methods on \textbf{Test 2}}
 \label{Tab2}
\end{table} \clearpage

\begin{table}[t] 
\centering
	\subfloat[GMRES after preconditioning with $P_{ad}$ and $P_{hy}$]{
	\begin{tabular}{| c | c | c | c | c | c | c |}
	\hline
	J   & \multicolumn{2}{ | c | }{ Iteration \# } & \multicolumn{2}{ | c | }{ CPU Time }& \multicolumn{2}{ | c | }{$\kappa_A$} \\
   	 & \multicolumn{2}{ | c | }{of GMRES} & \multicolumn{2}{ | c | }{ } & \multicolumn{2}{| c |}{ } \\ \hline
	NA & \multicolumn{2}{| c |}{554} & \multicolumn{2}{| c |}{14.3919} & \multicolumn{2}{| c |}{$4.0782 \times 10^3$} \\ \hline
	& $P_{ad}$ & $P_{hy}$ & $P_{ad}$ & $P_{hy}$ & $P_{ad}$ & $P_{hy}$ \\ \hline
	4   & 8 & 3 & 1.3422 & 1.1772 & 647.6787 & 615.1005 \\ \hline
	8   & 12 & 5 & 1.4953 & 1.2517 & 658.7462 & 627.0064 \\ \hline
	16 & 18 & 9 & 2.0216 & 1.5588 & 726.3005 & 690.1682 \\ \hline
	32 & 27 & 15 & 3.5402 & 2.4623 & 854.1450 & 788.5277 \\ \hline
	64 & 35 & 23 & 7.1266 & 4.9327 &  939.5190 & 816.1892 \\ \hline
	\end{tabular}}
	 \\
	\subfloat[Multiplicative Schwarz Iteration]{
	\begin{tabular} {| c | c | c | c | c |}
		\hline
		J & Iteration \# of & CPU Time& $ \| E_{mu} \|_A$ & $\rho(E_{mu})$ \\  
		  & Mult. Schwartz & & & \\ \hline
		4 & 2 & 1.1067 & 24.7399 & 0.0021 \\ \hline
		8 & 2 & 1.0982 & 24.6200 & 0.0526 \\ \hline
		16 & 3 & 1.1394 & 24.4247 & 0.2350 \\ \hline
		32 & 5 & 1.2778 & 24.2986 & 0.4369 \\ \hline
		64 & 7 & 1.6321 & 24.2524 & 0.5302 \\ \hline
	\end{tabular}}
	\caption{Performance of three Schwarz methods on \textbf{Test 3} } 
\label{Tab3}
\end{table} 

\begin{table}[b] 
\centering
	\subfloat[GMRES after preconditioning with $P_{ad}$ and $P_{hy}$]{
	\begin{tabular}{| c | c | c | c | c | c | c |}
	\hline
	J   & \multicolumn{2}{ | c | }{ Iteration \# } & \multicolumn{2}{ | c | }{ CPU Time }& \multicolumn{2}{ | c | }{$\kappa_A$} \\
   	 & \multicolumn{2}{ | c | }{of GMRES} & \multicolumn{2}{ | c | }{ } & \multicolumn{2}{| c |}{ } \\ \hline
	NA & \multicolumn{2}{| c |}{468} & \multicolumn{2}{| c |}{9.4276} & \multicolumn{2}{| c |}{$1.0769 \times 10^3$} \\ \hline
	& $P_{ad}$ & $P_{hy}$ & $P_{ad}$ & $P_{hy}$ & $P_{ad}$ & $P_{hy}$ \\ \hline
	4   & 8 & 2 & 1.3305 & 1.2039 & 103.5739 & 31.7538 \\ \hline
	8   & 11 & 2 & 1.4551 & 1.2558 & 75.7527 & 31.6954 \\ \hline
	16 & 14 & 3 & 1.8217 & 1.3019 & 56.6486 & 31.6803 \\ \hline
	32 & 13 & 5 & 2.2940 & 1.6227 & 46.4141 & 31.8710 \\ \hline
	64 & 15 & 8 & 3.7025 & 2.5950 & 44.1292 & 32.2846 \\ \hline
	\end{tabular}}
	\\
	\subfloat[Multiplicative Schwarz Iteration]{
	\begin{tabular} {| c | c | c | c | c |}
		\hline
		J & Iteration \# of & CPU Time& $ \| E_{mu} \|_A$ & $\rho(E_{mu})$ \\  
		  & Mult. Schwartz & & & \\ \hline
		4 & 2 & 1.0996 & 5.4574 & $85394 \times 10^{-9}$ \\ \hline
		8 & 2 & 1.0984 & 5.4575 & $1.0873 \times 10^{-6}$ \\ \hline
		16 & 2 & 1.1157 & 5.4575 & $6.6472 \times 10^{-4}$ \\ \hline
		32 & 2 & 1.1566 & 5.4560 & 0.0158 \\ \hline
		64 & 2 & 1.2399 & 5.4540 & 0.0678 \\ \hline
	\end{tabular}}
	\caption{Performance of three Schwarz methods on \textbf{Test 4} } 
\label{Tab4}
\end{table} \clearpage

\begin{figure}[t]
 \centering
 \subfloat[Plot of J vs. $\kappa_A(P_{ad})$]{
    \includegraphics[width=0.5\textwidth]{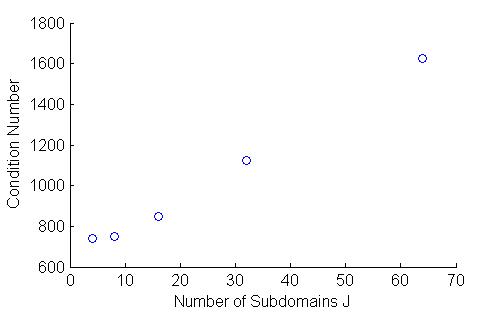}}
 \subfloat[Plot of J vs. $\kappa_A(P_{hy})$]{
    \includegraphics[width=0.5\textwidth]{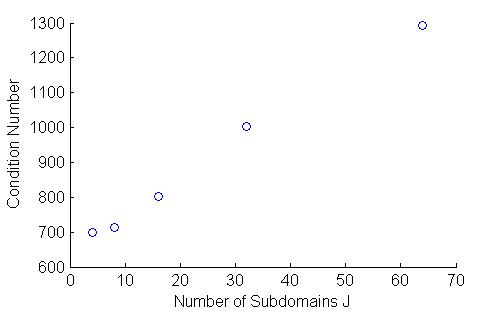}}
\caption{Dependence of $\kappa_A(P_{ad})$ and $\kappa_A(P_{hy})$ on J in \textbf{Test 2} \label{fig6.1}}
\end{figure}

\begin{figure}[b]
 \centering
 \subfloat[Plot of $\sigma(A)$]{
    \includegraphics[width=0.5\textwidth]{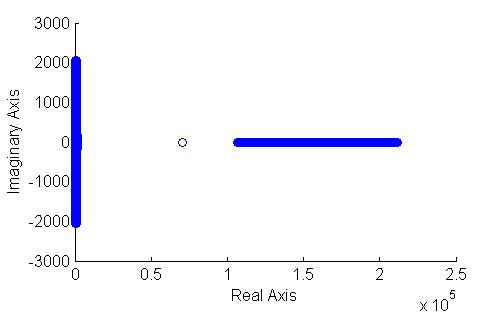}}
 \subfloat[Plot of $\sigma(P_{ad})$ with J = 4]{
    \includegraphics[width=0.5\textwidth]{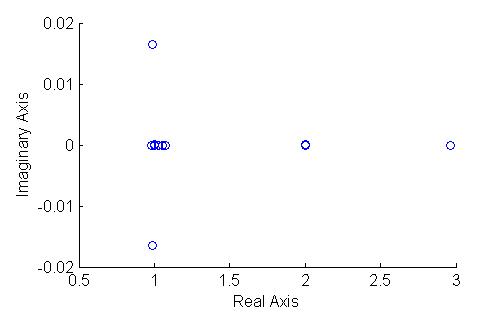}} \\
 \subfloat[Plot of $\sigma(P_{ad})$ with J = 8]{
    \includegraphics[width=0.5\textwidth]{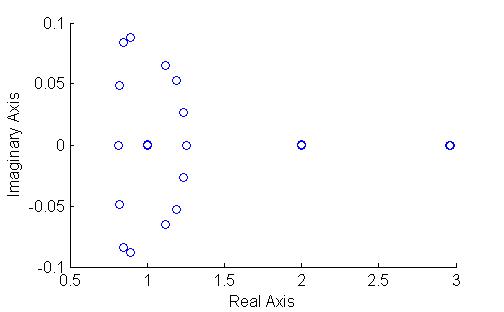}}
 \subfloat[Plot of $\sigma(P_{ad})$ with J = 16]{
    \includegraphics[width=0.5\textwidth]{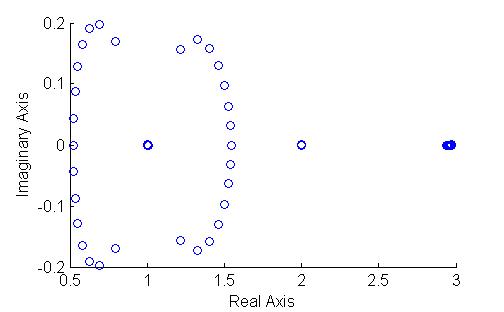}} \\
 \subfloat[Plot of $\sigma(P_{ad})$ with J = 32]{
    \includegraphics[width=0.5\textwidth]{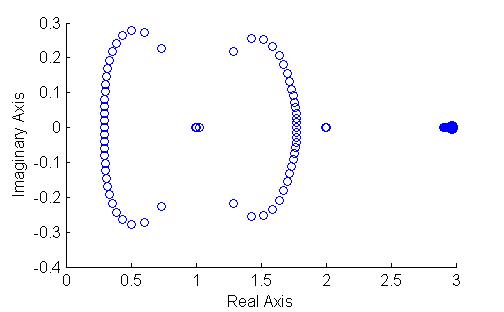}}
 \subfloat[Plot of $\sigma(P_{ad})$ with J = 64]{
    \includegraphics[width=0.5\textwidth]{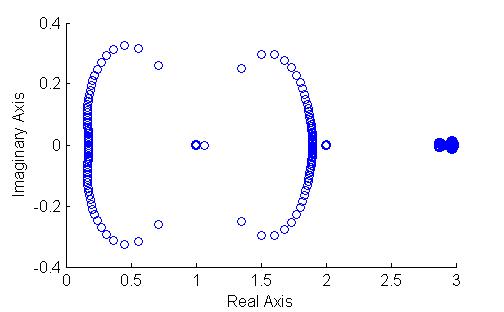}}
\caption{Spectrum plots from \textbf{Test 2} \label{fig6.2}}
\end{figure}  

\begin{figure}[ht]
 \centering
 \subfloat[Plot of $\sigma(A)$]{
    \includegraphics[width=0.5\textwidth]{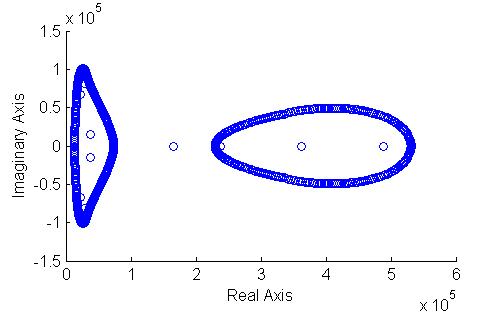}}
 \subfloat[Plot of $\sigma(P_{ad})$ with J = 4]{
    \includegraphics[width=0.5\textwidth]{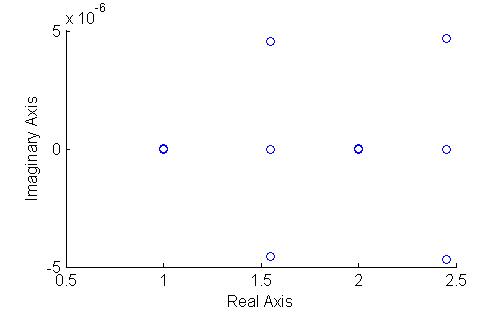}} \\
 \subfloat[Plot of $\sigma(P_{ad})$ with J = 8]{
    \includegraphics[width=0.5\textwidth]{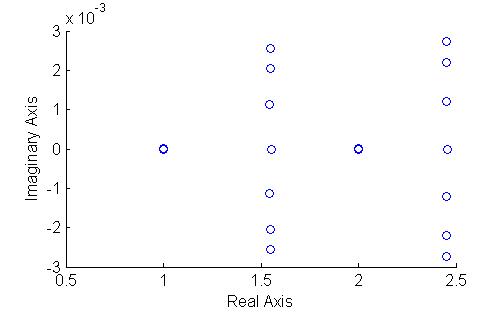}}
 \subfloat[Plot of $\sigma(P_{ad})$ with J = 16]{
    \includegraphics[width=0.5\textwidth]{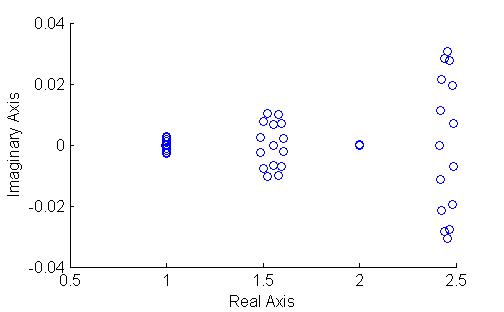}} \\
 \subfloat[Plot of $\sigma(P_{ad})$ with J = 32]{
    \includegraphics[width=0.5\textwidth]{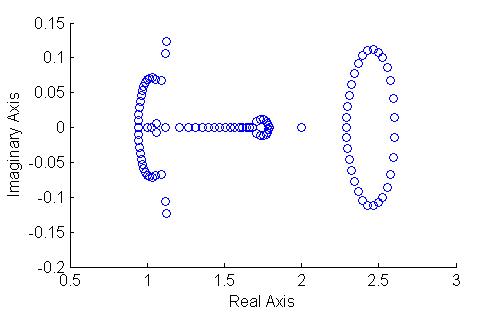}}
 \subfloat[Plot of $\sigma(P_{ad})$ with J = 64]{
    \includegraphics[width=0.5\textwidth]{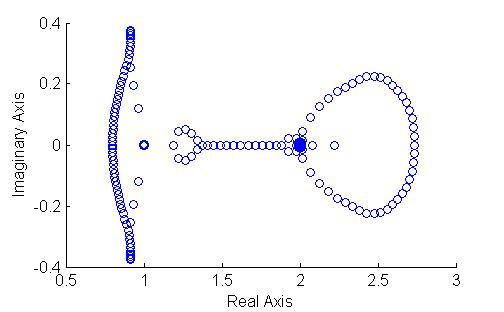}}
\caption{Spectrum plots from \textbf{Test 4} \label{fig6.3}}
\end{figure}  \clearpage

\begin{table}[ht] 
\centering
	\subfloat[\textbf{Test 1}]{
	\begin{tabular}{| c || c | c | c | c |}
		\hline
		\backslashbox{$1/H$}{$1/h$} & 16 & 32 & 64 & 128 \\ \hline \hline
		4 & 65.7321 & 421.0708 & $2.6132 \times 10^3$ & $1.5015 \times 10^4$ \\ \hline
		8 & 31.1342 & 182.6298 & $1.0812 \times 10^3$ & $5.2251 \times 10^3$ \\ \hline
		16 & 9.1835 & 66.5400 & 325.8230 & $1.3700 \times 10^3$ \\ \hline
	\end{tabular}} \\
	\subfloat[\textbf{Test 2}]{
	\begin{tabular}{| c || c | c | c | c |}
		\hline
		\backslashbox{$1/H$}{$1/h$} & 16 & 32 & 64 & 128 \\ \hline \hline
		4 & 29.9211 & 167.5776 & $1.0288 \times 10^3$ & $6.5720 \times 10^3$ \\ \hline
		8 & 16.9208 & 87.8272 & 567.0238 & $3.4546 \times 10^3$ \\ \hline
		16 & 7.9795 & 43.6354 & 249.6884 & $1.4125 \times 10^3$ \\ \hline
	\end{tabular}} \\
	\subfloat[\textbf{Test 3}]{
	\begin{tabular}{| c || c | c | c | c |}
		\hline
		\backslashbox{$1/H$}{$1/h$} & 16 & 32 & 64 & 128 \\ \hline \hline
		4 & 9.8123 & 28.7663 & 120.0237 & 598.6388 \\ \hline
		8 & 8.1720 & 18.9775 & 77.8493 & 413.6956 \\ \hline
		16 & 7.0895 & 15.4453 & 52.6716 & 279.7399 \\ \hline
	\end{tabular}} \\
	\subfloat[\textbf{Test 4}]{
	\begin{tabular}{| c || c | c | c | c |}
		\hline
		\backslashbox{$1/H$}{$1/h$} & 16 & 32 & 64 & 128 \\ \hline \hline
		4 & 7.7226 & 13.6618 & 27.8054 & 68.6090 \\ \hline
		8 & 7.2057 & 12.2264 & 23.0226 & 50.6394 \\ \hline
		16 & 6.9673 & 11.7750 & 21.6517 & 45.9328 \\ \hline
	\end{tabular}}
	\caption{Behavior of $\kappa_A(P_{ad})$ as $h$ and $H$ vary when
$J = 4$. \label{table 5.5}}
\end{table} \clearpage

\begin{table}[ht] 
\centering
	\subfloat[\textbf{Test 1}]{
	\begin{tabular}{| c || c | c | c | c |}
		\hline
		\backslashbox{$1/H$}{$1/h$} & 32 & 64 & 128 & 256 \\ \hline \hline
		8 & 212.7543 & $1.3522 \times 10^3$ & $6.7950 \times 10^3$ & $2.7336 \times 10^4$ \\ \hline
		16 & 75.7638 & 428.5503 & $1.8990 \times 10^3$ & $6.2384 \times 10^3$ \\ \hline
		32 & 11.0954 & 108.7492 & 441.1266 & $1.4728 \times 10^3$ \\ \hline
	\end{tabular}} \\
	\subfloat[\textbf{Test 2}]{
	\begin{tabular}{| c || c | c | c | c |}
		\hline
		\backslashbox{$1/H$}{$1/h$} & 32 & 64 & 128 & 256 \\ \hline \hline
		8 & 98.1017 & 669.2190 & $4.1498 \times 10^3$ & $2.3439 \times 10^4$ \\ \hline
		16 & 45.5532 & 306.0990 & $1.8395 \times 10^3$ & $8.5735 \times 10^3$ \\ \hline
		32 & 9.3610 & 111.3056 & 605.5322 & $2.6167 \times 10^3$ \\ \hline
	\end{tabular}} \\
	\subfloat[\textbf{Test 3}]{
	\begin{tabular}{| c || c | c | c | c |}
		\hline
		\backslashbox{$1/H$}{$1/h$} & 32 & 64 & 128 & 256 \\ \hline \hline
		8 & 17.5492 & 84.0599 & 470.5346 & $2.7253 \times 10^3$ \\ \hline
		16 & 12.2180 & 50.0108 & 314.2682 & $1.9264 \times 10^3$ \\ \hline
		32 & 8.2829 & 31.6447 & 183.8996 & $1.2169 \times 10^3$ \\ \hline
	\end{tabular}} \\
	\subfloat[\textbf{Test 4}]{
	\begin{tabular}{| c || c | c | c | c |}
		\hline
		\backslashbox{$1/H$}{$1/h$} & 32 & 64 & 128 & 256 \\ \hline \hline
		8 & 9.0739 & 16.7773 & 42.5919 & 147.1763 \\ \hline
		16 & 8.6294 & 15.0583 & 32.3845 & 107.8848 \\ \hline
		32 & 8.3903 & 14.4013 & 29.2256 & 91.4682 \\ \hline
	\end{tabular}}
	\caption{Behavior of $\kappa_A(P_{ad})$ as $h$ and $H$ vary when
$J = 8$. \label{table5.6}}
\end{table}

\par
\medskip
{\bf Acknowledgment:}  This work was initiated while the first autthor
was a long term visitor of the IMA at University of Minnesota in Fall 2010. 
The first author would like to thank the IMA for its financial support and 
its hospitality. The first author would also like to thank Blanca Ayuso
and Ohannes Karakashian for their many stimulating and critical discussions 
which have substantially helped to shape the structure and results of this paper.


\end{document}